\pdfoutput=1

\documentclass[11pt]{amsart}
\usepackage{Alegreya}

\newtheorem{theorem}{Theorem}[section]
\newtheorem{theorem*}{Theorem}
\newtheorem{lemma}[theorem]{Lemma}
\newtheorem{corollary}[theorem]{Corollary}
\newtheorem{definition}[theorem]{Definition}
\newtheorem{example}[theorem]{Example}

\newtheorem{proposition}[theorem]{Proposition}
\newtheorem{conjecture}[theorem]{Conjecture}
\theoremstyle{remark}
\newtheorem{remark}[theorem]{Remark}

\numberwithin{equation}{section}
\usepackage{enumerate}
\usepackage{tikz-cd}
\usetikzlibrary{bending}
\usepackage{graphicx,adjustbox}
\usepackage{hyperref}
\usepackage{paralist}
\usepackage{amsmath,amssymb}
\usepackage[normalem]{ulem}
\usepackage{upgreek}
\hypersetup{colorlinks,linkcolor={blue},citecolor={red},urlcolor={red}}
\usepackage[a4paper,
            bindingoffset=0.2in,
            left=1in,
            right=1in,
            top=1in,
            bottom=1in,
            footskip=.25in]{geometry}

\usepackage[
backend=biber,
bibstyle=numeric,
sorting=nyt,
]{biblatex}
\DeclareNameAlias{author}{last-first}
\DeclareFieldFormat[article, inbook]{title}{#1} 
\renewbibmacro{in:}{}
\DeclareFieldFormat{journaltitle}{#1\isdot}

\usepackage{amscd}
\usepackage{graphicx}
\usepackage[all]{xy}
\title[A Generalized Spectral Correspondence]{A Generalized Spectral Correspondence}
\author{Kuntal Banerjee}
\address{Department of Pure Mathematics, University of Waterloo, ON, Canada~ N2L 0A4}
\email{k4banerj@uwaterloo.ca}
\author{Steven Rayan}
\address{Centre for Quantum Topology and Its Applications (quanTA) and Department of Mathematics and Statistics, University of Saskatchewan, SK, Canada~ S7N 5E6}
\email{rayan@math.usask.ca}

\begin{document}
\begin{abstract}
We explore a strong categorical correspondence between isomorphism classes of sheaves of arbitrary rank on a given algebraic curve and twisted pairs on another algebraic curve, mostly from a linear-algebraic standpoint. In a particular application, we realize a generic elliptic curve as a spectral cover of the complex projective line $\mathbb{P}^1$ and then construct examples of cyclic pairs and co-Higgs bundles over $\mathbb{P}^1$. By appealing to a composite push-pull projection formula, we conjecture an iterated version of spectral correspondence. We prove this conjecture for a particular class of spectral covers of $\mathbb {P}^1$ through Galois-theoretic arguments. The proof relies upon a classification of Galois groups into primitive and imprimitive types. In this context, we revisit a classical theorem of Ritt.
\end{abstract}
\maketitle 
\tableofcontents

\let\thefootnote\relax\footnotetext{2020\textit{ Mathematics Subject Classification.} 14D20, 14H60, 20B35.}\let\thefootnote\relax\footnotetext{\textit{Keywords and phrases.} Spectral correspondence, spectral curve, twisted pair, Higgs bundle, co-Higgs bundle, moduli space, semistability, Hitchin fibration, projective line, elliptic curve, push-pull formula, cartographic group, Galois group, monodromy group.}

\section{Introduction}

The spectral correspondence for Higgs bundles, first identified by Hitchin \cite{Hitch2}, reveals finer geometric and algebraic structure within the moduli space of semistable $G$-Higgs bundles over a complex algebraic curve $X$.  Here, $G$ is a reductive group of finite rank and the so-called ``Higgs field'' of each Higgs bundle is valued in the canonical line bundle $K_X$. At the level of an individual Higgs bundle, the correspondence produces a new curve $\widetilde X$ encoding the spectrum of the Higgs field as a finite-to-one branched cover of $X$ together with a rank-$1$ sheaf on $\widetilde X$ that records the eigenspaces of the Higgs field.  When one pushes back this data to the original curve, the correspondence produces a representation of the original Higgs bundle in which the Higgs field is diagonalized everywhere save for at ramification points.  In this way, the spectral correspondence is a globalization of familiar aspects of the linear algebra of operators on finite-dimensional vector spaces --- in other words, of Higgs bundles over the point. The global spectral correspondence was subsequently expanded by Beauville-Narasimhan-Ramanan\cite{BNR} to the case of $L$-twisted Higgs fields, where $L$ is now an arbitrary line bundle. The corresponding moduli space of semistable $L$-twisted Higgs bundles was constructed by Nitsure \cite{Nitin}, who also gave a proof of properness of a morphism usually known in this context as the \textit{Hitchin morphism}.  The spectral correspondence, both in the original development for $L=K_X$ and for general $L$, leads to a convenient characterization of the generic fiber of the Hitchin morphism as the Jacobian (or Prym variety) of a spectral curve. A spectral correspondence between twisted pairs and vector bundles of higher rank was initiated by Hitchin and Schaposnik \cite{lau} in the setting of Higgs bundles associated to real subgroups of complex Lie groups. They refer to this operation as a ``nonabelianization'' of Higgs bundles. The operation is nonabelian in two related ways, first as the spectral bundle is no longer rank $1$ and in the fact that the fiber of the analogous Hitchin map is no longer an abelian variety. The spectral curve, according to their work, is the underlying reduced curve defined by a non-reduced characteristic polynomial.\\

In this article, we first provide some exposition about the algebro-geometric construction of spectral curves (Equation \ref{Spec}) and of the spectral correspondence for twisted Higgs bundles (also referred to as ``twisted pairs'') over a smooth algebraic curve. This exposition culminates in Theorem \ref{extension} and Corollary \ref{extsemi}), which emerge mostly through the language of linear algebra over unique factorization domains. We treat examples of cyclic Higgs bundles (Section \ref{stp}) and co-Higgs bundles of higher ranks (Section \ref{numcom}) over $\mathbb{P}^1$ in this way. Moreover, this perspective allows us to call upon Galois-theoretic techniques (Proposition \ref{Intermediate}) to elicit a threefold avatar of the spectral correspondence for cyclic pairs (\ref{finalth}).

\section{Overview of the article} 

Some inspiration for the algebraic construction of spectral curves in this article comes from \cite{Gallego,gallego2023higgs}. Recall that a smooth curve $X$ is a Noetherian scheme and each point $x\in X$ admits a Noetherian local ring (of stalks of regular functions). We replace the annihilating polynomials (Definition \ref{anni}) of twisted pairs with their counterparts on locally free stalks (Section \ref{annimor}) and explore linear maps on stalks over the function field of $X$ --- that is, over the stalk of regular functions at the generic point of $X$ --- culminating in theorems \ref{ufd1} and \ref{ufd2}. These linear maps on stalks furnish us with global characteristic polynomials of pairs and their invariant subbundles (Remark \ref{invariant}). We package this discussion ultimately as a sheaf theoretic correspondence between $X$ and a spectral curve $X_s$ that is embedded in the total space of the twisting line bundle (Theorem \ref{extension}). This further extends to a higher categorical correspondence between (semi)stable bundles on the spectral curve and (semi)stable pairs on $X$ (corollaries \ref{extsemi} and \ref{jorho} and Proposition \ref{har}).\\
    
 We investigate in Section \ref{Ite} a composite projection formula (\ref{itpush}) of locally free sheaves under composition of finite morphisms. This begs an immediate question about the factorizability of a smooth spectral covering map. We arrive at an affirmative answer in a foundational, yet ultimately nontrivial, case with the complex projective line as the base curve (Section \ref{compl}). We focus on a class of non-generic spectral curve that we call \textit{cyclic spectral curves}, owing to the fact that their Galois groups are cyclic (Section \ref{compl}). The Fundamental Theorem of Galois Theory and a categorical equivalence between function fields and algebraic curves (Proposition \ref{Intermediate}) are employed to complete the argument. We assemble the consequences of this so-called \textit{iterated spectral correspondence} in Theorem \ref{finalth}. We complete the article by revisiting a theorem of J.F. Ritt (Theorem \ref{Ritt}) that inspired us during the course of this investigation. We also pose a conjecture (\ref{conj}) that may lead to new directions in this research thread.\\

\let\thefootnote\relax\footnotetext{\textit{Keywords and phrases.} Spectral correspondence, spectral curve, twisted pair, Higgs bundle, co-Higgs bundle, moduli space, semistability, Hitchin fibration, projective line, elliptic curve, push-pull formula, cartographic group, Galois group, monodromy group.}

\noindent\textbf{Acknowledgements.} We thank Michael Gr\"ochenig for a thorough reading of a prior version of some of this material that appeared in the first-named author's doctoral dissertation (cf. \cite{Banerjee}), which was written under the supervision of the second-named author. This reading resulted in important recommendations and corrections.  During the preparation of this manuscript, the second-named author was partially supported by a Natural Sciences and Engineering Research Council of Canada (NSERC) Discovery Grant. The first-named author was supported by a Graduate Teaching Fellowship at the University of Saskatchewan that was funded in part by the second-named author’s Discovery Grant. Finally, both authors are grateful to the anonymous referee for their careful read of the manuscript, resulting in additional helpful suggestions and corrections.

\section{Background regarding $L$-twisted pairs on curves}

Let $X$ be an irreducible, nonsingular, projective algebraic curve over $\mathbb{C}$; equivalently, a smooth, compact, connected Riemann surface with genus $g_X\geq 0$. We will use ``curve'' and ``Riemann surface'' (or just \lq surface\rq) interchangeably to refer to such an object. For another convention and without ambiguity, we will use the symbol $1$ to denote the identity morphism from a bundle to itself (with the particular bundle understood by context).  Let $L$ be a holomorphic line bundle on $X$ with projection map $\pi: \text{Tot}(L)\to X$ where $\text{Tot}(L)$ is the total space of $L$. By an $L$-twisted \textit{pair} or \textit{Hitchin pair} on $X$ we mean a pair $(E, \phi)$ in which $E$ is a vector bundle over $X$ of finite rank $r$ and $\phi: E\to E\otimes L$ is a bundle morphism. The bundle morphism $\phi$ can be viewed as an element of $H^0(X, \text{End}(E)\otimes L)$. A \textit{morphism} of $L$-twisted pairs $(E, \phi)$ and $(E', \phi')$ is a commutative diagram as follows:

\begin{equation}
\begin{tikzcd}
E \arrow[r, "\phi"] \arrow[d, "\psi"]
& E\otimes L \arrow[d, "\psi \otimes 1 = \psi' "] \\
E' \arrow[r, "\phi' "]
& E'\otimes L
\end{tikzcd}
\end{equation} in which $\psi: E\to E'$ is a bundle morphism. Here, $1$ denotes the identity morphism on $L$ (consistent with our convention).  The pairs $(E, \phi)$ and $(E', \phi')$ are said to be \textit{isomorphic} if there exists an isomorphism $\psi: E\to E'$ of bundles such that $\phi' = \psi'\circ\phi\circ\psi^{-1}$. 

\begin{definition}
    A subbundle $F$ of an $L$-twisted pair $(E, \phi)$ is said to be $\phi$-\textbf{invariant} if $\phi(F)\subseteq F\otimes L$.
\end{definition}
\begin{definition}
    If $E$ is a vector bundle (over a curve), then its \textbf{slope} is the rational number$$\mu(E) = \displaystyle\frac{\deg(E)}{{\mbox{rank}}(E)}.$$
\end{definition}

\begin{definition}
An $L$-twisted pair $(E, \phi)$ is said to be  a \textbf{stable} (resp. \textbf{semistable}) pair if each nontrivial proper $\phi$-invariant subbundle $F$ satisfies the slope inequality 
\begin{equation}\label{slope}
\mu(F) < (\text{resp.}\leq)~ \mu(E).
\end{equation}
\end{definition}
\begin{remark}
    In the event that $E$ is stable (respectively, semistable), any $L$-twisted pair with underlying bundle $E$ is automatically stable (resp., semistable).
\end{remark}

\begin{proposition}
    Let $(E, \phi)$ be a semistable pair on $X$. Then, there exists a finite filtration of $\phi$-invariant subbundles of increasing ranks
    \begin{equation}
        0 = E_0\subset E_1\subset\dots \subset E_n = E
    \end{equation}
such that, for each $i = 1,\dots ,n$, we have $\mu(\frac{E_i}{E_{i - 1}}) = \mu(E)$ and the quotient pairs$$(\frac{E_i}{E_{i - 1}}, \phi_i:\frac{E_i}{E_{i - 1}}\to \frac{E_i}{E_{i - 1}}\otimes L )$$induced from $\phi$ are stable. The associated graded pair $\frak{gr}(E, \phi) = \bigoplus_{i = 1}^n(\frac{E_i}{E_{i - 1}}, \phi_i)$ is unique up to an isomorphism of $(E, \phi)$ and $\frak{gr}(E, \phi)$ is also semistable.
\end{proposition}

\begin{remark}\label{Jor-Hol}
    The above filtration is called a \textit{Jordan-H\"{o}lder filtration} of the pair $(E, \phi)$.  We call two $L$-twisted semistable pairs $(E, \phi)$ and $(E', \phi')$ S-\textit{equivalent} if their graded pairs $\frak{gr}(E, \phi)$ and $\frak{gr}(E', \phi')$ are isomorphic.
\end{remark}

Nitsure \cite{Nitin} established a moduli construction for $S$-equivalence classes of $L$-twisted pairs on $X$ through a GIT quotient by an action of $\text{GL}(N,\mathbb{C})$ or $\text{SL}(N, \mathbb{C})$ for sufficiently large $N$. By $\mathcal{M}(r, d, L)$, we denote the quasi-projective coarse moduli scheme of $S$-equivalent classes of $L$-twisted pairs so that the underlying bundle $E$ of each pair admits rank $r$ and degree $d$. This moduli scheme contains the scheme $\mathcal{M}'(r, d, L)$ of stable pairs as an open subscheme. Moreover, they established that the dimension of the Zariski tangent space of $\mathcal{M}'(r, d, L)$ satisfies the formula \begin{equation}
    \dim T_{(E, \phi)} = r^2\deg(L) + 1 + \dim H^1(X, L)
\end{equation} in each of the following cases: $L\cong K_X$; $L^r\ncong K_X^r$ but $\deg(K_X) = \deg(L)$; and finally $\deg(L) > \deg(K_X)$. Going forward, we will restrict to twisting line bundles $L$ of positive degree.

\section{Annihilating polynomials of pairs and the Hitchin morphism}\label{annimor}
In this context, we explore annihilating polynomials associated with a twisted bundle morphism. The characteristic polynomial of a pair is a specific example of an annihilating polynomial. Let $E$ and $L$ be a vector bundle and a line bundle, respectively, on a curve $X$. Now, fix a line bundle $L$ on $X$ and let $s = (s_1,\dots , s_n)\in\bigoplus_{i = 1}^n H^0(X, L^i)$. For $\phi\in H^0(X, \text{End}(E)\otimes L)$, consider$$\phi\otimes 1: E\otimes L^{i - 1}\to (E\otimes L)\otimes L^{i - 1} = E\otimes L^i.$$Then $\phi^i: E\to E\otimes L^i$ is defined by $\phi^i := (\phi\otimes 1)\circ\phi^{i - 1}$, with the convention $\phi^0 = 1$. This definition yields a global section of $\text{End}(E)\otimes L^n$, namely$$\phi^n + \sum_{i = 1}^n s_i\otimes \phi^{n - i}.$$
\begin{definition}\label{anni}
The polynomial $p(\lambda) = \lambda^n + \sum_{i = 1}^n s_i\lambda^{n - i}$, where $(s_1,\dots , s_n)\in \bigoplus_{i = 1}^n H^0(X, L^i)$, is said to be an \textbf{annihilating polynomial} of $\phi$ if 
\begin{equation}
\phi^n + \sum\limits_{i = 1}^n s_i\otimes \phi^{n - i} = 0.
\end{equation}
We will say that $\phi$ \textbf{satisfies} $p$ if $p$ annihilates $\phi$.
\end{definition}
In later sections, we will blend our main reasoning regarding spectral curves with arguments about the stalks of regular functions. In particular, we will relate restrictions of sheaves on open subsets and their germs. Let $\mathcal{O}_X$ denote the sheaf of regular functions on a curve $X$ (alternatively, the sheaf of holomorphic functions on a compact Riemann surface $X$). We elicit a parallel set of constructions of objects in the ambiance of linear algebra. The basic ingredients we use are the sheaf homomorphisms, over the Noetherian local ring of germs $\mathcal{O}_{X,x}$ at a point $x\in X$. Let $\mathcal{L}$ denote the sheaf of sections of $L$. Thus $\mathcal{L}^i$ is invertible for any integer $i$, that is, stalks are free of rank $1$ at each point $x$. Let $\phi\in H^0(X, \mbox{End}(E)\otimes L)$. For any open set $U$ of $X$ we restrict $\phi$ on $U$ as a sheaf homomorphism. Furthermore, choose a trivializing neighbourhood $U$ (assumed to be connected, if necessary) of $L^{-1}$ and let $\Lambda$ be a generator of the restricted sheaf $\mathcal{O}(L^{-1})|_{U}$. That means, for each open subset $V$ of $U$, the sheaf of sections of $L^{-1}$ on $V$ is generated by $\Lambda$. A tensor product yields an element $\phi\otimes\Lambda\in\mathcal{O}(\text{End}(E)\otimes L)(V)\otimes\mathcal{O}(L^{-1})(V)$. The set $\mathcal{O}(\text{End}(E)\otimes L)(V)\otimes\mathcal{O}(L^{-1})(V)$ is a natural subset of  $\mathcal{O}(\text{End}(E))(V)$ by definition and we indicate a local element $\psi = \phi\otimes\Lambda$  in the sheaf $\mathcal{O}(\text{End}(E))|_{U}$. Implementing the $\mathcal{O}_X$-isomorphism $\mathcal{O}(\text{End}(E))\cong \text{End}(\mathcal{O}(E))$ and and considering respective restrictions on open subsets of $U$, we identify $\psi$ as an $\mathcal{O}_X|_U$-endomorphism of $\mathcal{O}(E)|_{U}$. More generally, a global section $s_i$ of $L^i$ contributes natural elements in the sheaf $\mathcal{O}_X|_{U}$, namely, $a_i = s_i\otimes \Lambda^i$. Taking respective germs at $x\in U$, we write $\psi_x = \phi_x\otimes\Lambda_x$ and $a_{i, x} = s_{i, x}\otimes \Lambda_x^i$ simply taking tensor product over $\mathcal{O}_{X, x}$. Passing to the level of stalks $\phi_x\in \mathcal{O}(\text{End}(E)\otimes L)_x$ and $\Lambda_x$ denotes a generator of the $\mathcal{O}_{X, x}$-free module $\mathcal{L}_x^{-1}$. Note that the definition of $a_{i, x}$, as an element of the set $\mathcal{O}_{X, x}$, depends on the choice of $U$.  We view $\psi_x$ as an $\mathcal{O}_{X,x}$-linear endomorphism of the free module $\mathcal{O}(E)_x$. It is a routine to remark that an appropriate definition of an annihilating polynomial for $\psi$ and $\psi_x$ exists.\\

Going forward, we maintain the assumption that $U$ is a trivializing neighbourhood of $L^{-1}$ and the associated polynomials at a point $x$ are defined with respect to $U$. Let $(s_1,\dots , s_n)\in \bigoplus_{i = 1}^n H^0(X, L^i)$ be a chosen tuple and it defines polynomials $\lambda^n + \sum_{i = 1}^n a_i \lambda^{n - i}$ which admits coefficients from the restricted sheaf $\mathcal{O}_{X}|_U$ and $\lambda^n + \sum_{i = 1}^n a_{i,x} \lambda^{n - i}$ coefficients of which are contributed by the ring $\mathcal{O}_{X,x}$. We informally refer to such polynomials over $\mathcal{O}_X|_U$ (and over $\mathcal{O}_{X,x}$) as \textit{the associated polynomials} of $s = (s_1,\dots , s_n)$ on $U$. Here we remark that a polynomial $p(\lambda)$ (of the form $\lambda^n + \sum_{i = 1}^n s_i \lambda^{n - i}$) is an annihilating polynomial of $\phi$ precisely if the associated polynomial of $p$ on each trivializing neighbourhood $U$ of $L^{-1}$ annihilates $\psi$ and precisely if at each point $x$, the associated polynomial $\lambda^n + \sum_{i = 1}^n a_{i, x} \lambda^{n - i}$ annihilates $\psi_x$. We write, for $\phi\in H^0(\text{End}(E)\otimes L)$, its \textit{characteristic coefficients}, the global sections $s_i = (-1)^i\operatorname{tr}(\wedge^i\phi)\in H^0(X, L^i)$. The definition of these sections follows from the definition below. (The trace formulae of $\phi$ are computed pointwise.)
$$\operatorname{tr}(\wedge^i\phi) := \frac{1}{i!}\begin{vmatrix}
    \operatorname{tr}(\phi) & & i-1 & & 0 & & \hdots\\
    \operatorname{tr}(\phi^2) & & \operatorname{tr}(\phi) & & i-2 & & \hdots\\
    \vdots & & \vdots & & \vdots & & \hdots\\
    \operatorname{tr}(\phi^{i-1}) & & \operatorname{tr}(\phi^{i-2}) & & \hdots & & 1\\
    \operatorname{tr}(\phi^i) & & \operatorname{tr}(\phi^{i-1}) & & \hdots & & \operatorname{tr}(\phi)
\end{vmatrix}.$$

The characteristic polynomial of a twisted pair $(E, \phi)$ is $\lambda^r + \sum_{i = 1}^r (-1)^i\operatorname{tr}(\wedge^i\phi)\cdot\lambda^{n - i}$. Replacing the indeterminate symbol $\lambda$ with $\phi$ we obtain a global section of $\text{End}(E)\otimes L^r$, namely $\phi^r + \sum_{i = 1}^r s_i\otimes \phi^{r - i}$. Via the Cayley-Hamilton theorem, we arrive at 
\begin{equation}\label{cal}
    \phi^r + \sum\limits_{i = 1}^r s_i\otimes \phi^{r - i} = 0.
\end{equation} In particular, the characteristic coefficients of the local sheaf homomorphism $\psi$ are encoded here: 

\begin{equation}\label{assoeq}
(-1)^i\operatorname{tr}(\wedge^i\psi) = (-1)^i\operatorname{tr}(\wedge^i\phi)\otimes \Lambda^i;~
\end{equation} The respective coefficients for the stalk-wise homomorphism $\psi_x$ are analogous: 

\begin{equation}\label{assoeqx}
(-1)^i\text{tr}(\wedge^i\psi_x) = (-1)^i\text{tr}(\wedge^i\phi_x)\otimes \Lambda_x^i.
\end{equation}

In restricting our attention to twisted semistable pairs, we access the special role of the characteristic polynomials.  Consider a pair of integers $r>0$ and $d$. The map  
\begin{equation}\label{Hitchmor}
H: \mathcal{M}(r, d, L)\to \bigoplus\limits_{i = 1}^r H^0(X, L^i)
\end{equation}
that maps a semistable pair $(E, \phi)$ to its tuple of characteristic coefficients is a distinguished function --- called the \emph{Hitchin morphism} --- whose properness can be established by appealing to a valuative criterion \cite{Nitin}. The affine codomain of the map is accordingly called the \textit{Hitchin base}. The spectral correspondence that we explore in further sections describes the fibers of $H$. Given a tuple of sections $s = (s_1,\dots , s_r)\in \bigoplus_{i = 1}^r H^0(X, L^i)$, we supply an elementary example of a pair whose characteristic polynomial is defined by $s$. The origin of our example is a basic linear algebra question, which is whether one can find a matrix that realizes a given characteristic polynomial and, furthermore, whether the matrix can be constructed in a uniform way.  The answer is the so-called ``companion matrix''.

\begin{example}\label{semistable}
Let $E = \mathcal{O}\oplus L^{-1}\oplus\dots \oplus L^{-(r - 1)}$. The bundle morphism $\phi$ that is the \textit{companion matrix} of $s = (s_1,\dots , s_r)\in \bigoplus_{i = 1}^r H^0(X, L^i)$ is\begin{equation}
    \phi = \begin{bmatrix}
0 & 0 &\dots &\dots & -s_r\\
1 & 0 &\dots &\dots  & -s_{r - 1}\\
0 & 1 & 0 &\dots  & -s_{r - 2}\\
\vdots & \vdots & \ddots &\ddots &\vdots\\
0 & 0 &\dots & 1 & -s_1
\end{bmatrix}.
\end{equation}
\end{example}

We may rapidly confirm the stability of such a pair in the case $r = 2$. Let $M$ be an invariant sub-line bundle of $E$. Then the holomorphic projection $\pi_1: M\to \mathcal{O}$ is a nonzero bundle map. Thus, $\deg(M^*\otimes L^{-1})\geq 0$ and $\mu(M) < \mu(E)$. For any $r$, stability of such a pair was proved for $g_X\geq 2$ and $L = K_X$ (cf. Remark 3.8 in \cite{Hausel2021VerySH}).

\section{Construction of spectral curves}
The following constructions are adopted from \cite{BNR}. Replacing a complex curve with a smooth, connected projective curve defined over a general field, we present the general construction of the spectral curves. Let $\Bbbk$ be an algebraically closed field of characteristic $0$ and $X$ is a smooth irreducible projective algebraic curve over $\Bbbk$. For any tuple sections $s = (s_1,\dots , s_n)$ as in Section \ref{annimor}, we construct a $1$-dimensional scheme $X_s$ embedded in the total space of line bundle $L$. We describe $X_s$ in two different ways interchangeably. The first definition is according to section 3 in \cite{BNR} which we modify into a convenient form. Recall that the tautological line bundle $\pi^*L$ over $L$ admits a tautological section $\eta$ defined by $\eta(y) = (y, y)\in \pi^*L$ for a point $y$ on any fiber of $L$. More rigorously, $\eta\in \mbox{Hom}_{\mathcal{O}_X}(\mathcal{L}^{-1}, \mbox{Sym}(\mathcal{L}^{-1}))$ is the canonical morphism, equivalently a global section of $\pi^*L$ by the adjunction formula of the pullback and the pushforward operations. The scheme $X_s$ is defined as the intersection of the polynomial $\lambda^n + {\scriptstyle{\sum_{i = 1}^n}} s_i \lambda^{n - i}$ with $\eta$, that is, 
\begin{equation}\label{spe}
X_s =\left\{y\in L: \eta^n(y) + \sum\limits_{i = 1}^n (\pi^*s_i\otimes\eta^{n - i})(y) = 0\right\}.  
\end{equation}
 The second construction is more abstract and decidedly more algebro-geometric in nature. Recall that $\mathcal{L}$ denotes the sheaf of sections of $L$. For each $i = 1,\dots , n$ there is a sheaf homomorphism defined by multiplication of section $s_i:\mathcal{O}(L^{-n}) = \mathcal{L}^{-n}\to \mathcal{L}^{-(n - i)} = \mathcal{O}(L^{-(n - i)})$. For the sake of completeness, we denote with $s_0$ the identity morphism on $\mathcal{L}^{-n}$. Taking the sum of these maps yields a sheaf homomorphism $\bigoplus_{i = 0}^n s_i:\mathcal{L}^{-n}\to \bigoplus_{i = 0}^n\mathcal{L}^{-i}\subset \text{Sym}(\mathcal{L}^{-1})$. Further, $\mathcal{I}$ denotes the ideal sheaf generated by the image of $\bigoplus_{i = 0}^n s_i$. The sheaf $\frac{\text{Sym}(\mathcal{L}^{-1})}{\mathcal{I}}$ over $X$ is a quasi-coherent sheaf of algebra. We define, as in Section 3 in \cite{BNR}:
\begin{equation}\label{Spec}
X_s = \text{Spec}\left(\frac{\text{Sym}(\mathcal{L}^{-1})}{\mathcal{I}}\right).
\end{equation}
In the standard language of the schemes, the definition in \ref{spe} describes the set of closed points of the scheme defined in definition \ref{Spec}. We observe that $X_s\subset \mbox{Spec}\left(\mathcal{L}^{-1}\right) = \mbox{Tot}(L)$ is a closed subscheme.
\begin{definition}
We refer to the scheme $X_s$ as the \textbf{spectral curve} associated to a point $s$ in the Hitchin base $\oplus_{i = 1}^n H^0(X, L^i)$. The restriction of the bundle map $\pi$ on $X_s$ is a finite morphism, called the \textit{spectral covering map} of $X$.
\end{definition}

At times, we also use the more general and dimension-insensitive terminology \emph{spectral cover}.

\begin{remark}
Over $\Bbbk = \mathbb{C}$, the collection of all tuples $s = (s_1,\dots , s_n)\in \bigoplus_{i = 1}^n H^0(X, L^i)$ for which the scheme $X_s$ is smooth is an open dense subset in case $L^n$ has no base points. Note that the set of spectral curves forms a complete linear system of divisors of $\pi^*L^n$ over $L$. By Bertini's theorem smooth divisors form a Zariski open subset in the projective completion of the linear system. On the other hand, the branch points of the finite morphism $\pi$ are given by the resultant of the polynomial $\lambda^n +\sum_{i = 1}^n s_i\lambda^{n - i}$ and its $\lambda$-derivative $n\lambda^{n - 1} +\sum_{i = 1}^{n - 1} (n - i)s_i\lambda^{n - i -1}$ which is a global section of $L^{n(n-1)}$. A point on $X$ is a branch point of $\pi$ if and only if it is a zero of the resultant. Away from such points $\pi$ is \'{e}tale of degree $n$. The set of sections $s = (s_1,\dots , s_n)$ such that underlying resultant section admits distinct zeros is Zariski open. The locus in the Hitchin base consisting of such spectral curves is called the \textit{smooth locus of spectral curves}. More specifically, we call a smooth, integral spectral curve a \textit{generic spectral curve}. In case of an integral, smooth spectral curve, the morphism $\pi: X_s\to X$ is finite (so proper) with the degree $r$. For the rest of the article we will assume that a spectral curve is smooth, thus a reduced scheme.
\end{remark}
    
For a trivializing neighbourhood $U$ of $L^{-1}$, we produce the sheaf of ideals restricted on $U$ as $\mathcal{I}|_U = \left < \sum_{i = 0}^na_i\Lambda^{n - i} \right >$. We recall that $\mbox{Sym}(\mathcal{L}^{-1}) \cong \bigoplus_{i = 0}^\infty \mathcal{L}^{-i}$ and furthermore $$\text{Sym}(\mathcal{L}^{-1})|_U = \mathcal{O}_X|_U[\Lambda] = \left\{\sum_{i = 0}^kf_i\Lambda^{k - i}: f_i\in\mathcal{O}_X|_U;~ k\geq 0\right\}.$$

Now, consider any two trivializing neighbourhoods $U$ and $V$ with local generators $\Lambda$ and $\upmu$ of $\mathcal{L}^{-1}$. Then $\Lambda = g_{VU}\upmu$ for some nonvanishing $\Bbbk^*$ valued function $g_{VU}$ on $U\cap V$. Restrictions of this equality also make sense on open subsets of $U\cap V$. Let us choose and fix a transition data between $\mathcal{I}|_U$ and $\mathcal{I}|_V$ (and between $\mbox{Sym}(\mathcal{L}^{-1})|_U$ and $\mbox{Sym}(\mathcal{L}^{-1})|_V$) such that the division algorithm is a sheaf homomorphism (recall that $\mathcal{O}_X$ is a sheaf of commutative rings with identity and the division algorithm under division by a monic polynomial is available at our service). By gluing the local sheaf homomorphisms we obtain a global sheaf homomorphism.
\begin{equation}\label{aniso}
    \frac{\mbox{Sym}(\mathcal{L}^{-1})}{\mathcal{I}} \cong \mathcal{O}\oplus\mathcal{L}^{-1}\oplus\dots \oplus \mathcal{L}^{-(n-1)}
\end{equation}
Let $\Lambda_x$ be the germ of a generator of the stalk of the sheaf $\mathcal{L}^{-1}$ at $x$. Then we have 

$$\mbox{Sym}(\mathcal{L}^{-1})_x = \left\{\sum\limits_{i = 0}^kf_i\Lambda_x^{k - i}: f_i\in \mathcal{O}_{X, x};~ k\geq 0\right\}.$$ In other words, $\mbox{Sym}(\mathcal{L}^{-1})_x = \mathcal{O}_{X, x}[\Lambda_x]$. Likewise, we have an explicit description of the germ of ideal $\mathcal{I}$ at $x$ as a principal ideal
\begin{equation}\label{imageideal}
    \mathcal{I}_x = \left<\sum\limits_{i = 0}^na_{i, x}\Lambda_x^{n - i}\right>.
\end{equation}

As per the definition of $X_s$, the finite morphism $\pi$ gives a sheaf isomorphism $\pi_*\mathcal{O}_{X_s}\cong \frac{\text{Sym}(\mathcal{L}^{-1})}{\mathcal{I}}$ (cf. p.128, Exercise 5.17 in \cite{Hart}) and the Euler characteristic of $\mathcal{O}_{X_s}$ is given as $$\chi(X_s, \mathcal{O}_{X_s}) = \chi(X, \pi_*\mathcal{O}_{X_s}) = \chi\left(X, \bigoplus\limits_{i = 0}^{n - 1} \mathcal{L}^{-i}\right).$$ From Riemann-Roch theorem this leads to \begin{equation}
    \chi(X_s, \mathcal{O}_{X_s}) = -\frac{n(n-1)}{2}\cdot\deg(L) + n(1 - g_X).
\end{equation}
We compute the genus of $X_s$ by a formula that appears, for instance in Section 3 in \cite{BNR}:
\begin{equation}\label{genus}
   g_{X_s} = 1 - \chi(X_s, \mathcal{O}_{X_s}) = \deg(L)\cdot\frac{n(n-1)}{2} + n(g_X - 1) + 1.
\end{equation}

\section{Spectral correspondence for generic spectral covers}\label{twist}
There is, of course, an intimate connection between the irreducibility of a spectral curve and the factorizability of its underlying defining polynomial. We assume here that a spectral curve is reduced. We observe that the Cauchy product of polynomial sections $\lambda^m +\sum_{i = 1}^m s_i\lambda^{m - i}$ and $\lambda^n +\sum_{j = 1}^n t_j\lambda^{n - j}$ is a polynomial section with coefficients in the Hitchin base $\bigoplus_{i = 1}^{m + n} H^0(X, L^i)$. A reduced spectral curve is an irreducible scheme exactly when its underlying spectral polynomial can not be factored into a product of spectral polynomials of smaller degrees (cf. \cite{Haupau} Lemma 2.4). A local description of the sheaf $\mathcal{I}\subset\mbox{Sym}(\mathcal{L}^{-1})$ by the polynomials in the local generator $\Lambda$ fully captures the features of an integral spectral curve.  We observe that the spectral curve $X_s$ is integral if and only if for each nonempty open subset $V$ of $X$ the ring $\frac{\mbox{Sym}(\mathcal{L}^{-1})}{\mathcal{I}}(V)$ is an integral domain. Furthermore, we restrict our attention to the open subsets of a trivializing neighbourhood: if $V$ is an open subset of a trivializing neighbourhood $U$ then the ring $\mathcal{I}(V)$ which is generated by a polynomial in $\Lambda(V)$ is irreducible over $\mathcal{O}_X(V)$, in case the spectral curve is integral. We pass the description to the level of stalks at individual points and summarize the explanation in the following proposition.
\begin{proposition}\label{integrals}
The scheme $X_s$ is integral if and only if the associated polynomial at each point of $X$ is irreducible over $\mathcal{O}_{X,x}$.  
\end{proposition}

\begin{remark}
Above Proposition \ref{integrals} supplies an argument for the fact that the Krull dimension of a smooth integral spectral curve is $1$. We observe this fact by computing the Krull dimension of $\frac{\mbox{Sym}(\mathcal{L}_x^{-1})}{\mathcal{I}_x}$ at each point $x\in X$. We denote the local generator (which we have already introduced) with $\Lambda$ and realize that the ring $\mathcal{O}_{X, x}[\Lambda_x]$ has Krull dimension $2$. This statement follows from Hilbert's basis theorem that states that a polynomial ring over a Noetherian ring is Noetherian. The Krull dimension decreases by $1$ after taking quotient by the prime ideal $<\sum_{i = 0}^n a_{i,x}\Lambda_x^{n - i}>$. We finally implement Krull's height theorem which suggests that a principal prime ideal of a Noetherian ring has height $1$.
\end{remark}

Our next explanations, which set the stage for a higher-rank spectral correspondence, will rely on the linear algebra of stalks at individual points. We briefly remind the reader about two relevant properties of unique factorization domains and their quotient fields of fractions. Let $R$ be a unique factorization domain and $F$ be its field of fractions.

\begin{theorem}\label{ufd1}
    Let $f\in R[x]$ be a primitive polynomial and $g\in R[x]$. Then $f$ divides $g$ in $F[x]$ if and only if $f$ divides $g$ in $R[x]$.
\end{theorem}

\begin{theorem}\label{ufd2}
    Let $f\in R[x]$ be a polynomial of degree $n\geq 1$. Then $f$ is a product of two polynomials in $F[x]$ of degrees $d$ and $e$ respectively with $0< d, e < n$ if and only if there exist polynomials $g, h\in R[x]$ of degrees $d$ and $e$ respectively with $0< d, e < n$ such that $f = g\cdot h$.
\end{theorem}

Now, let $(E, \phi)$ be an $L$-twisted pair on $X$ that admits $\lambda^r +\sum_{i = 1}^r s_i\lambda^{r - i}$ as its characteristic polynomial, defined by $s$. As usual, the stalks and the germs are defined with respect to a trivializing neighbourhood. The module $\mathcal{O}(E)_x$ is free over $\mathcal{O}_{X,x}$ of rank $r$ and the module homomorphism $\psi_x$ extends linearly on the $K$-vector space (of same rank) $V(x) = K\otimes_{\mathcal{O}_{X,x}} \mathcal{O}(E)_x$, where $K$ is the quotient field of $\mathcal{O}_{X,x}$. Here $K$ is isomorphic to the function field of the nonsingular curve $X$. We view the characteristic polynomials of the module homomorphisms of at the level of stalks same as the respective characteristic polynomials of respective $K$-endomorphisms on $V(x)$. If there is a proper invariant subbundle of $(E, \phi)$ then at each point $x$, the characteristic polynomial of the corresponding germ $\psi_x$ is divisible by the characteristic polynomial of the germ, say $\psi_x'$, contributed by the invariant subbundle where the coefficients of the quotient polynomial are contributed by the quotient field $K$. By the Gauss' lemma, the divisibility over $K$ descends to the divisibility over UFD $\mathcal{O}_{X,x}$. This polynomial $\psi_x'$ has a strictly smaller degree over $\mathcal{O}_{X, x}$. Thus, we arrive at a contradiction that $X_s$ is integral as in Proposition \ref{integrals}. In fact, we can pose a more explicit factorization. In this formulation, the characteristic polynomial of $(E, \phi)$ is written as the product of the characteristic polynomials of $(F, \phi_{|F})$ and the induced quotient pair $(E/F, \phi_{E/F})$. Thus, we have the following conclusion.
\begin{remark}\label{invariant}
If $X_s$ is an integral scheme then a twisted pair $(E, \phi)$ with characteristic polynomial defined by $s$ does not admit any nontrivial proper invariant subbundle $(F, \phi_{|F})$, so automatically stable.
\end{remark}

We are now in the position to prove a categorical equivalence between torsion-free sheaves over $X_s$ and $\mathcal{O}_X$-locally free $\frac{\text{Sym}(\mathcal{L}^{-1})}{\mathcal{I}}$-modules over $X$. This is also framed as a one-to-one correspondence between isomorphism classes of vector bundles over the reduced spectral curve and isomorphism classes of twisted pairs over $X$. Consider a smooth and integral spectral cover $X_s$. We denote the corresponding spectral polynomial by $p_s$, which we will simply refer to as $p$ in case the choice of $s$ is obvious from the context. The associated polynomials are $p', p'_x$ while the latter is irreducible over the UFD $\mathcal{O}_{X, x}$ and thus irreducible over the function field $K$.\\

We begin with a locally free sheaf $M$ of rank $n$ with the multiplication operation by the tautological section $\eta: M\to M\otimes\pi^*L$. This is pushed forward to a $L$-twisted Hitchin pair $(\pi_*M, \pi_*\eta)$ on $X$ using projection formula. On the spectral curve, section $\eta^r + \sum_{i = 1}^r \pi^*s_i\otimes\eta^{r - i}$ vanishes, thus the pair $(\pi_*M, \pi_*\eta)$ of rank $nr$ satisfies $p$ under pushforward operation $\pi_*$ by the \textit{adjoint} operation of pulling back sections (p.33, Proposition 4.2 in \cite{bre}). This is same as giving $\pi_*\mathcal{O}_{X_s}$-structure on $\pi_*M$. So $p$ is an annihilating polynomial of $\pi_*\eta$. This establishes one of the two sides of the correspondence presented in form of Theorem \ref{extension}. Moreover, we obtain $p^n$ as the characteristic polynomial of the pair.\\

At a point $x$, the element $\psi_x$ satisfies $p'_x$. Consider the minimal polynomial $g$ of $\psi_x$ over $K$. Then $g$ divides $p'_x$ over $K$. If $\deg(g)<\deg(p'_x)$, then $p'_x$ is reducible over $K$ which contradicts the case. Thus, $\deg(g) = \deg(p'_x)$ and monic polynomial $p'_x$ is the minimal polynomial over $K$. Recall that the irreducible factors of the minimal polynomial and the characteristic polynomial of an endomorphism (on a finite dimensional vector space) coincide. It follows that $p_x'^n$ is the characteristic polynomial of the $K$-linear map $\psi_x$. (Implicitly, we have that $p_x'^n$ is the characteristic polynomial of $\psi_x$ over $\mathcal{O}_{X, x}$.) Denoting the coefficients of $p^n$ by $S_i$'s and the coefficients of $p_x'^n$ by $A_i$'s, we have $A_{i} = (-1)^i\text{tr}(\wedge^i\psi_x)$. Here, the associated polynomials of $p^n$ are $p'^n$. Thus $A_i = S_{i, x}\otimes\Lambda_x^i$. Appealing to formula \ref{assoeqx}, we have $S_{i, x} = (-1)^i\text{tr}(\wedge^i\phi_x)$ and this holds for any $x\in X$. This leads to the conclusion that $S_i = (-1)^i\text{tr}(\wedge^i\phi)$, so $p^n$ is the characteristic polynomial of $\phi$.\\

On the other hand, let $p$ be an annihilating polynomial of $(E,\phi)$. Invoking the arguments of the previous paragraph, we may conclude that the rank of $E$ is divisible by $r = \deg(p)$. Furthermore, if the rank of $E$ is $nr$ for some $n\geq1$, then the characteristic polynomial of $(E,\phi)$ must be $p^n$. Finally, if $F$ is an invariant subbundle of such a pair $(E,\phi)$, so that $(F,\phi|_F)$ is a twisted pair in its own right, then we note that the characteristic polynomial of $(F,\phi|_F)$ is $p^k$ for some $k\leq n$.\\

Now, let us consider a trivializing neighbourhood $U$ of $L^{-1}$ and an action on $\mathcal{O}(E)|_U$ given as $$\alpha|_U: \mbox{Sym}(\mathcal{L}^{-1})|_U\to \mathcal{O}(\mbox{End}(E))|_U\cong \mbox{End}(\mathcal{O}(E))|_U$$ with $q\mapsto q(\psi)$. Here, for each open subset $V$ of $U$, we treat $\psi$ as a sheaf homomorphism on $\mathcal{O}(E)|_V$. We claim that $\ker(\alpha|_U) = \mathcal{I}|_U$. If $f\in \mathcal{I}|_U(V)$ then we have $f(\psi) = 0$. So $f\in \ker(\alpha|_U)(V)$. Now consider $f\in \ker(\alpha|_U)(V)$. Let $x\in V$ be a point. Taking the germs at $x\in V$ we have $f_x(\psi_x) = 0$. The minimal polynomial of $\psi_x$ is $p'_x$ as we regard $\psi_x$ as a linear map over $K$. Then $f_x$ is divisible by $p'_x$ over $K$ and over $\mathcal{O}_{X, x}$ because $p'_x$ is monic, so primitive. Then we use the division algorithm. We divide $f$ by the restriction of $p'$ defined on $V$, as elements of $\mathcal{O}_X(V)[\Lambda_V]$. There are unique elements $g$ and $h$ such that $f = p'.g + h$ and $\deg(h) < \deg(p')$. But, for each $x\in V$ there is a neighbourhood $W_x$ such that $f$ is divided by the polynomial $p'$ over $W_x$ due to the divisibility of the germs. From the uniqueness feature of the division algorithm over $\mathcal{O}_{X, x}$ we have the germ $h_x = 0$. This is true for each $x\in X$, thus $h(x) = 0$ for all $x\in V$. Thus, we have $h = 0$ on $V$. Thus $f\in \mathcal{I}$ and $\ker(\alpha|_U) = \mathcal{I}|_U$ from the set theoretic equality on each open subset $V$ of $U$. Finally, let $f_1 - f_2\in \ker(\alpha|_U)(V) = \mathcal{I}|_U(V)$. The action of $f_1$ and $f_2$ on $\psi$ is invariant. It confirms that there is a well-defined action $\alpha|_U$ by $\left(\mbox{Sym}(\mathcal{L}^{-1})/\mathcal{I}\right)|_U$ on $\mbox{End}(\mathcal{O}(E))|_U$. The action is compatible over all the trivializing neighbourhoods of $L^{-1}$ over $X$ and defines a global action on $\mathcal{O}(E)$ by $$\alpha: \frac{\text{Sym}(\mathcal{L}^{-1})}{\mathcal{I}}\to \text{End}(\mathcal{O}(E)).$$

The main ingredient of the spectral correspondence is the following categorical isomorphism induced via the pushforward morphism $\pi_*$ (p.128, Exercise 5.17 in \cite{Hart}).
\begin{remark}\label{corres}
   The pushforward operation $\pi_*$ defines a categorical equivalence between the category of quasi-coherent sheaves of $\mathcal O_{X_s}$-modules and the category of $\mathcal O_X$-quasi-coherent sheaves that admit a $\pi_*\mathcal O_{X_s}$-module structure. In particular, the isomorphism classes of quasi-coherent sheaves of $\mathcal O_{X_s}$-modules on $X_s$ are in correspondence with the isomorphism classes of $\pi_*\mathcal O_{X_s}$-modules on $X$.
\end{remark}

We now observe that twisted pairs on $X$ admit $\pi_*\mathcal O_{X_s}$-module structures so that their underlying locally free sheaves $E$ are pushforward sheaves of locally free sheaves on $X_s$. In particular, $\mathcal{O}(E)$ can be written as the pushforward of a sheaf of sections of a vector bundle $M$ over $X_s$ of rank $n$. Suppose that $p$ annihilates $(E_1, \phi_1)$ and $(E_2, \phi_2)$. Then $(E_1, \phi_1)\cong (E_2, \phi_2)$ if and only if $E_1$ and $E_2$ are isomorphic as $\mbox{Sym}(L^{-1})/\mathcal I$-modules. The injectivity of the correspondence that is inherent in \ref{corres} can now be rephrased in the following way: the locally free sheaf $M$ is obtained from the $\pi_*\mathcal{O}_{X_s} = \mbox{Sym}(\mathcal{L}^{-1})/\mathcal{I}$-module structure induced by $\pi_*\eta$ on $\pi_*M$. In the same spirit, we denote by $M$ the associated unique vector bundle over $X_s$.\\

We conclude that $\pi_*(M, \eta) = (E, \phi)$. Let $\mathcal{F} = \mathcal{O}(\pi^*L)$. Then there is a multiplication map $-\eta:\mathcal{F}^{-1}\to \mathcal{O}_{X_s}$ and the sheaf homomorphism $1\oplus (-\eta):\mathcal{F}^{-1}\to \mathcal{F}^{-1}\oplus\mathcal{O}_{X_s}$ defines an ideal $\mathcal{G}\subset \mbox{Sym}(\mathcal{F}^{-1})$. By implementing the division algorithm on the restricted sheaf over a trivializing neighbourhood we obtain $\frac{\mbox{Sym}(\mathcal{F}^{-1})}{\mathcal{G}}\cong \mathcal{O}_{X_s}$. Hence, the $\mathcal{O}_{X_s}$-locally free sheaf structure on $M$ is isomorphic to the structure induced by $\frac{\mbox{Sym}(\mathcal{F}^{-1})}{\mathcal{G}}$ via the algebra morphism $\frac{\mbox{Sym}(\mathcal{F}^{-1})}{\mathcal{G}}\to \mbox{End}(M)$ defined as $q\mapsto q(\eta.I_M)$, where $I_M$ stands for the identity moprhism on $M$. (Indeed, any scalar multiple of the identity morphism is annihilated by a linear polynomial.) That $M\cong M'$ if and only if $(M, \eta) \cong (M', \eta)$ supports the fact that the action by the tautological section $\eta$ on the sheaf $M$ does not change the $\mathcal O_{X_s}$ module structure of $M$. Also, $\eta$ satisfies, as a bundle morphism, Equation \ref{spe} which defines $X_s$. So, a $\frac{\text{Sym}(\mathcal{L}^{-1})}{\mathcal{I}}$-structure on $\mathcal{O}(E)$ is obtained as pushforward of a $\frac{\text{Sym}(\mathcal{F}^{-1})}{\mathcal{G}}$-structure of $M$ --- that is, $\pi_*(M, \eta) = (E, \phi)$. We coalesce this whole discussion into the following theorem.\\

\begin{theorem}\label{extension}
Let $X_s$ be a nonsingular, integral spectral curve over $X$ with finite (so proper) covering map $\pi$. Then there is a one-to-one correspondence between isomorphism classes of vector bundles $M$ of a finite rank over $X_s$ and $L$-twisted Hitchin pairs $(E, \phi)$ over $X$ annihilated by $p_s$. The correspondence is given by $(M, \eta)\mapsto (\pi_*M, \pi_*\eta)$ using the projection formula $$\pi_*(M\otimes \pi^*L)\cong \pi_*M\otimes L.$$
\end{theorem}

Let $F\subset E$ be an invariant subbundle. Write $E = \pi_*M$ and $F = \pi_*N$ where $N\subset M$ as a locally free subsheaf. To show that $M/N$ is locally free we use an $\mathcal{O}_X$-isomorphism $\pi_*(M/N)\cong\pi_*M/\pi_*N = E/F$. The latter being locally free, $\pi_*(M/N)$ is locally free so a coherent sheaf $M/N$ is indeed locally free due to the above correspondence. We mention this fact in the following remark.
\begin{remark}\label{subbunpair}
    The correspondence preserves subbundles of bundles on $X_s$ and invariant twisted subbundles of twisted pairs on $X$. 
 \end{remark}  
    
\begin{remark}\label{spec1}
In the classical case of $n = 1$ in Theorem \ref{extension}, $p$ is the characteristic polynomial of the underlying pairs on $X$. 
\end{remark}

As is well known, the generic fiber of the Hitchin morphism (\ref{Hitchmor}) for a fixed rank and degree is the Jacobian of the spectral curve $X_s$. This corresponds with the case $n = 1$ in \ref{extension}. In Hitchin's work, we see that the locally free sheaf $M$, for any value of $n$, is explicitly determined as $\ker(\eta\cdot I-\pi^*\phi)\otimes \mathcal{L}'$ for a fixed invertible sheaf $\mathcal{L}'$, cf. Proposition 5.17 \cite{Hausel2021VerySH}. In case $n\geq 2$ we should not directly use the term ``fiber'' (because we can not duplicate the Hitchin morphism with a morphism which may capture the coefficients of an annihilating polynomial); rather, we observe that the space of S-equivalence classes of $L$-twisted semistable pairs of a given rank and a degree which are annihilated by $p$ is a scheme for a generic choice of $p$.  This scheme is represented, in the case of $g_{X_s} > 1$, by the moduli space of $S$-equivalence classes of the semistable bundles on $X_s$ for a fixed rank and a fixed degree. It has the structure of an irreducible projective algebraic variety that contains the moduli space of isomorphism classes of stable bundles as an open smooth subvariety.

\begin{corollary}\label{extsemi}
In Theorem \ref{extension}, $M$ is a stable (resp. semistable) bundle on $X_s$ if and only if $(E, \phi)$ is a stable (resp. semistable) $L$-twisted pair on $X$.
\end{corollary}
\begin{proof} In this context we recall that the degree of the pushforward bundle (more generally for push forward of quasi-coherent sheaves) is given as 
\begin{equation}\label{pushdeg}
    \deg(\pi_*M) = \deg(M) + \text{rank}(M)(1 - g_{X_s}) -\deg(\pi)\text{rank}(M)(1 - g_X).
\end{equation}
A proof of this formula is modeled on the proof of Proposition 4.3 in \cite{bre}. The key observation here is that an argument that works for the pushforward of a line bundle will make sense for vector bundles of arbitrary ranks. The rest follows from Remark \ref{subbunpair} and the immediate fact that a subbundle $N$ of $M$ obeys the slope inequality if and only if the subbundle $f_*N$ of $f_*M$ obeys the slope inequality \ref{slope}.
\end{proof}

\begin{corollary}\label{jorho}
     Let $M$ be a semistable bundle on a smooth integral spectral curve $X_s$ of rank $n$. Consider a Jordan-H\"older filtration $0 = M_0\subset M_1\subset\dots \subset M_k = M$ that is, $\mu(M) = \mu(M_i/M_{i - 1})$ and $M_i/M_{i - 1}$ is stable for all $i$. Then a Jordan-H\"older filtration of $\pi_*(M, \eta)$ (as defined in \ref{Jor-Hol}) is given by $0 = \pi_*(M_0, \eta)\subset \pi_*(M_1, \eta)\subset\dots \subset\pi_*(M_k, \eta) = \pi_*(M, \eta)$. Conversely, let $(E, \phi)$ be a semistable pair on $X$ which is annihilated by $p$ and a Jordan-H\"older filtration $0 = (E_0, \phi)\subset\dots \subset (E_k, \phi) = (E, \phi)$ is obtained as the pushforward of a filtration of $M$ such that $\pi_*(M, \eta) = (E, \phi)$. Finally, $\frak{gr}(E, \phi) \cong \pi_*\frak{gr}(M, \eta)$.
\end{corollary}
The justification is immediate: since $\pi$ is a sufficiently well-behaved morphism --- in particular, since is finite --- the pushforward operation commutes with quotients and the direct sum of bundles. In what follows, we extend this equivalence to a Harder-Narasimhan filtration of bundles and pairs.

\begin{proposition}\label{har} The following statements are true:

    \textbf{A:} Let $E$ be a vector bundle over a curve $X$. Then, $E$ has a unique increasing filtration by vector subbundles $ 0 = E_0\subset E_1\subset E_2\subset\dots \subset E_k = E$ for which $gr_i = E_i/E_{i - 1}$ satisfies the following conditions:
    
    (i) the quotient $gr_i$ is semistable;
    
    (ii) $\mu(gr_i) > \mu(gr_{i + 1})$ for $i = 1,\dots , k - 1$.
    
    \textbf{B:} Likewise, let $(E, \phi)$ be an $L$-twisted pair on $X$. Then, $E$ has a unique increasing filtration by invariant subbundles $ 0 = E_0\subset E_1\subset E_2\subset\dots \subset E_k = E$ such that the quotient pair $gr_i = (E_i/E_{i - 1}, \phi_i)$ satisfies the following conditions:
    
    (i) the quotient $gr_i$ is a semistable pair;
    
    (ii) $\mu(gr_i) > \mu(gr_{i + 1})$ for $i = 1,\dots , k - 1$.
\end{proposition}
There is a straightforward proof of (B) which is identical to the one with the vector bundle case in (A). We follow a proof given for vector bundles, available in Lemma 5.6 \cite{Daly}, Proposition 5.7 \cite{Daly}, Lemma 5.8 \cite{Daly}, and Proposition 5.9 \cite{Daly}. A detailed discussion on boundedness-related results is also available in the literature (\cite{Potier}).

\begin{corollary}
    The Harder-Narasimhan filtrations of pairs $(E, \phi)$ on $X$ are in one-to-one correspondence with the Harder-Narasimhan filtrations of bundles $M$ on the smooth integral spectral cover $X_s$.
\end{corollary}

\section{Stable pairs and spectral curves on $\mathbb{P}^1$}\label{stp}
We shift our focus to specific base curves here in order to make concrete observations about the spectral correspondence. We will focus in particular on $\mathbb{P}^1$ as the base, as many objects here can be written down explicitly and completely. Let $t \geq 2$. We denote by $\pi:\mathcal{O}(t)\to \mathbb{P}^1$ the unique line bundle of degree $t$ admitting a holomorphic transition data $z\mapsto z^t$ on the set of nonzero complex numbers. We will rely repeatedly on the classical algebro-geometric fact that every holomorphic bundle $E$ on $\mathbb{P}^1$ determines a sequence of integers $m_1\geq\dots \geq m_r$, unique up to permutation such that, 
\begin{equation}\label{Gro}
E\cong \mathcal{O}(m_1)\oplus\cdots\oplus\mathcal{O}(m_r).  
\end{equation}
Of course, $r$ here is the rank of $E$.  A holomorphic bundle map $\phi: E\to E\otimes \mathcal{O}(t)$ is said to be a $t$-twisted endomorphism on $\mathbb{P}^1$. Enabled by the so-called \emph{Birkhoff-Grothendieck splitting} in \ref{Gro}, we may adopt a global representation of $\phi$ as an $r \times r$ matrix. The $(i,j)$-th entry is a section $\phi_{i,j}\in H^0(\mathbb{P}^1, \mathcal{O}(m_i - m_j + t))$ that denotes the component wise map between $\mathcal{O}(m_j)\to\mathcal{O}(m_i + t)$. Recalling that the corresponding tautological line bundle $\pi^*\mathcal{O}(t)$ over $\text{Tot}(\mathcal{O}(t))$ admits a canonical section $\eta$, the spectral curve defined by sections $s = (s_1,\dots , s_r)$ of $\mathcal{O}(t),\dots , \mathcal{O}(tr)$ respectively, is the curve  
 \begin{equation}\label{defspec}
 X_s = \left\{y\in\mbox{Tot}(\mathcal{O}(t)): \eta^r(y) + s_1(\pi(y))\eta^{r-1}(y) +\dots  + s_r(\pi(y)) = 0\right\}.
 \end{equation}
 It is the zero scheme of a global section of the line bundle $\pi^*\mathcal{O}(tr)$. A closer look at the space $\text{Tot}(\mathcal{O}(t))$ gives a clearer understanding of spectral curves. We first realize $\mathbb{P}^1$ as the complex space $(\mathbb{C}\bigsqcup \mathbb{C})/\Phi$ where $\Phi$ is a biholomorphism $\Phi:\mathbb{C}^*\to\mathbb{C}^*; \Phi(x) = \frac{1}{x}$. The space $\text{Tot}(\mathcal{O}(t))$ is then realized as $((\mathbb{C} \times \mathbb{C})\bigsqcup (\mathbb{C} \times \mathbb{C}))/\Psi$ through identifying the open subset $\mathbb{C}^* \times \mathbb{C}$ with itself by the biholomorphism $\Psi(x, y) = \left(\frac{1}{x}, \frac{y}{x^t}\right)$ (cf. p.39 in \cite{bre}). The underlying bundle map $\pi$ takes $[(x,y)]$ to $[x]\in\mathbb{P}^1$. The space of global holomorphic sections of a line bundle $\mathcal{O}(ti)$ is characterized by the complex polynomials of degree $\leq ti$. Thus, we write a pair of complex affine curves simplifying the definition of a spectral curve as following 
 \begin{equation}
 \begin{cases}\label{spec}
     y^r + s_1(x)y^{r-1} +\dots + s_r(x) = 0;\\
     \tilde{y}^r + \tilde{s}_1(\tilde{x})\tilde{y}^{r-1} +\dots + \tilde{s}_r(\tilde{x}) = 0
\end{cases}
\end{equation}

We construct the spectral curve $X_s$ by patching together the two complex affine curves defined above along the biholomorphic map $\Psi$.  We caution that a spectral curve is a complex analytic space that may not always be a Riemann surface and which may admit singular points. In any case, the spectral covering morphism $[(x, y)] \mapsto [x]$ is a finite morphism between complex analytic spaces. To analyze a spectral curve and its covering map, it suffices to focus on one of the two affine patches in Equation \ref{spec}.
\begin{lemma}\label{lem1}
A smooth spectral curve over $\mathbb{P}^1$ is integral if and only if one of two affine curves is irreducible (as a complex polynomial in two variables).
\end{lemma}
\begin{proof} If the spectral curve $X_s$ is integral, then the affine curves are irreducible simultaneously on the standard open neighbourhoods. If not, we would then obtain a factorization of the associated polynomials. On the other hand, suppose that $$y^r + s_1y^{r - 1} +\dots + s_r = (y^{r_1} + u_1y^{r_1 - 1} +\dots + u_{r_1})(y^{r_2} + v_1y^{r_2 - 1} +\dots + v_{r_2})$$ on one of two affine coordinate charts, meaning that all the coefficients $u_i$ and $v_j$ are elements of $\mathbb{C}[x]$. Here $0 < r_1, r_2 < r$. It is enough to show that $\deg(u_i)\leq ti$ and $\deg(v_j)\leq tj$ to show that polynomial on the other chart is reducible. Changing the coordinates with $\Psi$, we obtain a factorization $$\tilde{y}^r + \tilde{s}_1\tilde{y}^{r - 1} +\dots + \tilde{s}_r = (\tilde{y}^{r_1} + \tilde{u}_1\tilde{y}^{r_1 - 1} +\dots + \tilde{u}_{r_1})(\tilde{y}^{r_2} + \tilde{v}_1\tilde{y}^{r_2 - 1} +\dots + \tilde{v}_{r_2}).$$ Note that the left hand side is a monic polynomial with coefficients over UFD $\mathbb{C}[\tilde{x}]$ and the coefficients $\tilde{u}_i$'s and $\tilde{v}_j$'s are elements of $\mathbb{C}(\tilde{x})$. This is possible precisely when $\tilde{u}_i$'s and $\tilde{v}_j$'s are elements of $\mathbb{C}[\tilde{x}]$ that is, $\deg(u_i)\leq ti$ and $\deg(v_j)\leq tj$ for all $i, j$. The spectral curve $X_s$ is both reduced and irreducible since the irreducible affine algebraic curves are, because the quotient rings (obtained by quotienting with the ideals generated by these polynomials) are integral domains.
\end{proof}

We highlight a specific class of examples of non-generic points, in the sense that their discriminant sections do not necessarily admit distinct zeros, on the affine base. Let $s\in H^0(\mathbb{P}^1, \mathcal{O}(tr))$ be a section with distinct zeros over $\mathbb{P}^1$. The set of such sections is a Zariski open subset of the affine space $H^0(\mathbb{P}^1, \mathcal{O}(tr))$ so we call such elements as the \textit{generic sections} of $\mathcal{O}(tr)$.

\begin{definition}\label{cyd}
Let $s\in H^0(\mathbb{P}^1, \mathcal{O}(tr))$. We call a $t$-twisted pair $(E, \phi)$ \textbf{cyclic} if it admits characteristic polynomial $\lambda^r - s$. In case $s$ is a generic element that is, admits simple roots we call $(E, \phi)$ a \textbf{generic cyclic pair} and corresponding spectral curve a \textbf{generic cyclic spectral curve}. 
\end{definition}
\begin{remark}
    A partial justification for above the definition will be given in section \ref{compl}.
\end{remark}
Our Definition \ref{cyd} generalizes the ordinary cyclic Higgs bundles associated to cyclic quivers (cf. \cite{Ste2}). Observe that the spectral polynomial $\lambda^r - s$ is not in the smooth locus for $r > 2$. It is only for $r = 2$ that we have a spectral curve belonging to the smooth locus. This is due to the fact that its discriminant is a nonzero constant multiple of $s^{r - 1}$. Smooth, integral cyclic spectral covers are characterized in Remark 3.1 \cite{BNR} and Remark 3.5 \cite{BNR}. The generic cyclic spectral covers are integral. Choose a generic section $s$ of the line bundle $L^r$. So, we can not write this divisor in form of $m\cdot D$ for some divisor $D$ on $X$ with $m > 1$ dividing $r$. Moreover, these generic cyclic spectral covers are smooth due to the Jacobian criterion of smoothness. We mention another algebraic proof of the same fact in this context.

\begin{lemma}\label{irred}
     A generic cyclic spectral curve is integral on $\mathbb{P}^1$, and so a generic cyclic pair is stable for all $r\geq 2$.
\end{lemma}   
\begin{proof} It suffices to prove that the polynomial $y^r - s(x)$ is irreducible over the unique factorization domain $\mathbb C[x]$. Let $x_0$ be a root of $s$. Then $s(x)$ lies inside the prime ideal $P=\langle x - x_0\rangle$ but not inside $P^2$ since $x_0$ is not a repeated root of $s$. It follows from Eisenstein's criterion that $y^r - s(x)$ is irreducible.\end{proof}

The next theorem restricts the Grothendieck numbers (i.e. the degrees of the summand line bundles of the underlying bundle over $\mathbb{P}^1$) of the semistable Hitchin pairs on $\mathbb{P}^1$.
 
\begin{theorem}\label{stgro}
 Let $E \cong \mathcal{O}(m_1)\oplus\dots \oplus \mathcal{O}(m_r)$ be a vector bundle over $\mathbb{P}^1$ such that the integers $m_i$ satisfy $m_1\geq\dots \geq m_r$. Let $t\geq0$ be an integer. If $E$ is the underlying bundle of a $t$-twisted semistable Hitchin pair then, for $1\leq i\leq r - 1$, we have
\begin{equation}\label{inq}
m_{i}\leq m_{i + 1} + t.\\
\end{equation}
Let us suppose that $E$ obeys inequality \ref{inq} and take $s\in H^0(\mathbb{P}^1, \mathcal{O}(tr))$. Then, there exists a cyclic pair $(E, \phi_s)$ with characteristic coefficients $(0,..,0, s)$.
\end{theorem}
\begin{proof} The proof of necessary condition \ref{inq} previously appeared in theorems 3.1 and 6.1 of \cite{Rayan,Ste1}, respectively. As such, we prove the latter part of the theorem and, within that, we focus on certain extreme cases, suggesting a strategy that applies to all cases. Since the integers $m_1,\dots ,m_r$ satisfy the simultaneous inequalities$$0\leq m_2 - m_1 + t \leq t,\;\dots ,\;0\leq m_r - m_{r - 1} + t \leq t$$as well as the identity$$(m_2 - m_1 + t) +\cdots + (m_r - m_{r - 1}  + t) + (m_1 - m_r + t) = tr,$$we arrive at$$t\leq m_1 - m_r + t \leq tr.$$Now, choose a section $s\in H^0(\mathbb{P}^1, \mathcal{O}(tr))$. In the case that$$m_2 - m_1 + t =\cdots = m_r - m_{r - 1} + t = 0,$$we have the equality $m_1 - m_r + t = tr$. This leads us to the construction of $\phi_s$ as 
$$\phi_s = \begin{bmatrix}
0 & 0 &\dots &\dots & \pm s\\
1 & 0 &\dots &\dots  & 0\\
0 & 1 & 0 &\dots  & 0\\
\vdots & \vdots & \ddots &\ddots &\vdots\\
0 & 0 &\dots & 1 & 0
\end{bmatrix},$$
\noindent where we have adjusted signs as necessary. At the other extreme, we have $$m_2 - m_1 + t =\cdots = m_r - m_{r - 1} + t = t$$and $m_r - m_1 + t = t$.  We can of course represent $s$ on an affine chart by a complex polynomial of degree at most $tr$. We exploit the Fundamental Theorem of Algebra to distribute its roots in a strategic way, wherever permitted in the global components. We write $s = u_1\dots u_r$ with each factor having degree at most $t$ (and where one or more $u_i$ may be $1$). Then, we construct $\phi_s$ --- again, adjusting signs as necessary --- in the following way:
$$\phi_s = \begin{bmatrix}
0 & 0 &\dots &\dots & \pm u_r\\
u_1 & 0 &\dots &\dots  & 0\\
0 & u_2 & 0 &\dots  & 0\\
\vdots & \vdots & \ddots &\ddots &\vdots\\
0 & 0 &\dots & u_{r - 1} & 0
\end{bmatrix}.$$ To prove the statement in the remaining cases, the general strategy is suggested by the latter extreme case: we factor the determinant $s$ and carefully regroup the irreducible factors as needed and distribute them in the matrix.
\end{proof}
It may be worth noting that cyclic pairs play a special role in the geometry of twisted Higgs bundle moduli spaces because they can be used to define, in the case of $t$-twisted pairs, analogues of the so-called ``Hitchin section'' (cf. \cite{Ste2} for example) for usual Higgs bundle moduli spaces. Moreover, the space of cyclic chains is acted upon by an $(r - 1)$-real-dimensional compact group as follows. Let $E$ be a bundle on $\mathbb{P}^1$ and $(u_1,\dots , u_r)$ be sections on $\mathbb{P}^1$ as in our last proof. If $u_i = 0$ for some $1\leq i\leq r-1$ then $E_i\cong\mathcal{O}(m_1)\oplus\cdots\oplus\mathcal{O}(m_i)$ is invariant and $\mu(E_i)\geq \mu(E)$. To obey the stability property, we restrict $u_i\neq 0$ for $1\leq i\leq r-1$. On the other hand, let $u_i\neq 0$ for $1\leq i\leq r-1$ and $u_r\neq 0$. Then, there is no nonzero proper invariant subbundle and stability is automatic. If $u_r = 0$, then a nonzero proper invariant subbundle is either of $\mathcal{O}(m_r);\dots ;~\mathcal{O}(m_r)\oplus\dots \oplus\mathcal{O}(m_1)$, so that semistability of this pair is respected. Finally, we restrict $\deg(E)$ and $\text{rank}(E)$ to be mutually prime to confirm stability of each semistable cyclic pair. Now, we may define an natural action of the $(r-1)$-fold product  $\mathbb{S}^1\times\dots \times\mathbb{S}^1$  on $(1,\dots ,1)$-cyclic chains: \begin{align*}
(\lambda_1,\dots , \lambda_{r - 1}).\begin{bmatrix}
0 & 0 &\dots &\dots & \phi_r\\
\phi_1 & 0 &\dots &\dots  & 0\\
0 & \phi_2 & 0 &\dots  & 0\\
\vdots & \vdots & \ddots &\ddots &\vdots\\
0 & 0 &\dots & \phi_{r - 1} & 0
\end{bmatrix}\\ 
= \begin{bmatrix}
0 & 0 &\dots &\dots & \lambda_1^{-1}\dots \lambda_{r - 1}^{-1}\phi_r\\
\lambda_1\phi_1 & 0 &\dots &\dots  & 0\\
0 & \lambda_2\phi_2 & 0 &\dots  & 0\\
\vdots & \vdots & \ddots &\ddots &\vdots\\
0 & 0 &\dots & \lambda_{r - 1}\phi_{r - 1} & 0
\end{bmatrix}.
\end{align*}
If two such chains are equivalent under the group action, they are isomorphic as pairs. Indeed $(E, (\lambda_1,\dots ,\lambda_{r - 1}).\phi) \cong (E, \phi)$ since $(\lambda_1,\dots ,\lambda_{r - 1}).\phi = \psi\phi\psi^{-1}$ where $\psi$ denotes a diagonal matrix with $i$-th diagonal entry $\lambda_1\dots \lambda_{i-1}$. We denote the underlying orbit space on $\mathbb{P}^1$ by $\mathcal{M}(m_1,\dots , m_r, t)$ keeping in mind that $\sum m_i$ is co-prime to $r$ and describe the moduli as the quotient $$\frac{\prod\limits_{i = 1}^{r - 1} \mathbb{C}^{m_{i + 1} - m_i + t + 1}\backslash\{0\}\times \mathbb{C}^{m_1 - m_r + t + 1}}{(\mathbb{S}^1)^{r - 1}}.$$ We observe that this is a proper group action by a compact Hausdorff group. Thus, this quotient is a Hausdorff space, and the orbits are closed real submanifolds of the parent space.

\section{Numerical computation of the co-Higgs sheaves}\label{numcom}
On $\mathbb{P}^1$, an $\mathcal{O}(2)$-twisted pair is normally referred to as a \textit{co-Higgs bundle} (\cite{Rayan,Ste1}). In this case, the generic spectral curve is an elliptic curve (cf. Equation \ref{genus}) embedded in the total space of $\mathcal{O}(2)$. The spectral correspondence for semistable rank $2$ co-Higgs bundles was described completely, at the level of an algebraic equation for each Hitchin fiber including the singular ones, thereby producing in turn a specific algebraic realization of the entire moduli space as a specific quasiprojective variety residing in a given ambient space \cite{Ste1}. Consistent with the theme and goals of the present article, we explore the spectral correspondence for co-Higgs bundles of higher rank. We recall some results (\cite{Ati}) at this stage in order to compute pushforward bundles of vector bundles on an elliptic curve. Examples of a similar computation for rank $1$, utilizing the push-pull projection formula, can be found in \cite{Ste1}. In the following computations, we work with a complex elliptic curve $X$ realized as a $2:1$ branched covering map $f:X\to\mathbb{P}^1$ and $E$ will denote an indecomposable bundle of rank $n$ over $X$.\\

\textbf{Case I:} Let $\deg(E) = 0$. Then, the degree of the bundle $f_*E$ on $\mathbb{P}^1$ is $\deg(f_*E) = -2n$. Appealing to the Birkhoff-Grothendieck decomposition, we will write$$f_*E\cong \mathcal{O}(a_1)\oplus\cdots \oplus \mathcal{O}(a_{2n}),$$where$$a_1 +\cdots + a_{2n} = -2n.$$Here $\dim H^0(X,E) = 0$ or $1$ as per Lemma 15 in \cite{Ati}.\\

(a) If $\dim H^0(X,E) = 0$ then $a_i < 0$ for all values of $i$. Otherwise, we have $\dim H^0(X,E) \geq 1$. Thus, $a_i \leq - 1$ holds $\forall i$ and $\sum_{i = 1}^{2n} a_i \leq - 2n$. In fact, equality holds if and only if $a_i = - 1$, for all values of $i$. Thus, 
\begin{equation}
f_*E\cong \mathcal{O}(-1)\oplus\cdots \oplus \mathcal{O}(-1).
\end{equation}

(b) If $\dim H^0(X,E) = 1$ then there is an $i$ such that $a_i\geq 0$. If there are $i, j$ with $a_i\geq 0$ and $a_j\geq 0$, then $\dim H^0(X,E)\geq 2$. So, $a_i\geq 0$ for only one value of $i$ say $i = 1$. On the other hand, $a_1 > 0\implies \dim H^0(X,E) > 1$, so $a_1 = 0$. Setting $a_1 = 0$, we have
$$a_2 +\dots + a_{2n} = -2n$$with $a_2 < 0,\dots , a_{2n} < 0$. From $a_2\leq -1,\dots , a_{2n}\leq -1$ obtain that$$a_2 +\dots + a_{2n}\leq -2n + 1$$and finally there is exactly one value $i\geq 2$ --- say, $i = 2$ such that $a_2 = -2$. Thus, 
\begin{equation}
f_*E \cong \mathcal{O}\oplus \mathcal{O}(-2)\oplus \mathcal{O}(-1)\oplus\dots \oplus \mathcal{O}(-1).
\end{equation}

$\textbf{Case II:}$ Let $\deg(E) = 1$. From Lemma 15 \cite{Ati}, $\dim H^0(X,E) = 1$. Let suppose$$f_*E \cong \mathcal{O}(a_1)\oplus\dots \oplus\mathcal{O}(a_{2n}).$$ From the degree computation of the pushforward bundle, we have $\deg(f_*E) = 1 - 2n$. From earlier reasoning $a_1 = 0$ and thus $a_2 +\dots + a_{2n} = 1 - 2n$. On the other hand,$$a_2\leq -1,\dots , a_{2n}\leq -1.$$This leads us to $ a_2 =\dots = a_{2n} = -1$. Finally, 
\begin{equation}
f_*E \cong \mathcal{O}\oplus \mathcal{O}(-1)\oplus\cdots\oplus \mathcal{O}(-1).
\end{equation}

\begin{lemma}
Let $E$ be an indecomposable vector bundle over an elliptic curve $X$ such that $\deg(E) < 0$. Then $H^0(X,E)$ is trivial.
\end{lemma}
\begin{proof} Suppose that $\deg(E) = d <0$. As $E$ is indecomposable, $E^*$ is also indecomposable, and so we may apply Lemma 15 of \cite{Ati} to obtain$$\dim H^0(X, E^*) = \deg(E^*) = -d.$$By Serre duality in combination with the triviality of $K_X$, we have $\dim H^1(X, E) = -d.$ Finally, by Riemann-Roch, we have$$\dim H^0(X, E) - \dim H^1(X, E) =\deg(E) = d\;\implies\;\dim H^0(X, E) = 0.$$
\end{proof}
\textbf{Case III:} Let $\deg(E) = -1$. Then, $\dim H^0(X, E) = 0$. On the other hand, $\deg(f_*E) = -1 - nr$. Write $f_*E \cong \mathcal{O}(a_1) \oplus\dots \oplus
\mathcal{O}(a_{2n})$ with $a_1\leq -1,\dots , a_{2n}\leq -1$. From the argument we have used in previous cases, we have unique $i$, say $1$, such that $a_1 = -2$ and others $a_2,.., a_{2n}$ are $-1$. Thus, we produce 
\begin{equation}
f_*E \cong \mathcal{O}(-2)\oplus\mathcal{O}(-1)\oplus\dots \oplus \mathcal{O}(-1).
\end{equation}

There is no general computational strategy immediately available to us if $E$ is a bundle of rank larger than $1$ and if $d$ is a number outside of $0, 1, -1$. We find it useful to consider rank $2$ semistable bundles in this context as well as a covering map $f$ of degree $2$. A semistable bundle can be either indecomposable or decomposable. An indcomposable bundle on an elliptic curve is semistable (cf. \cite{Tu}) and a decomposable bundle is of the form $E\cong L_1 \oplus L_2$ in which $L_1$ and $L_2$ share the same degree. If $E$ is indecomposable, then it suffices to treat the cases where $E$ admits degree in $\{-1, 0, 1, 2\}$.  Should $E$ admit any other degree, we simply adjust our computation by a push-pull projection. We are now left with the only case where $E$ has degree $2$ and we recall some relevant definitions and techniques from \cite{Ati} for this purpose:

\begin{definition}
    Let $X$ be a smooth projective algebraic curve; $E$, a vector bundle of rank $r$ on $X$; and$$0 = E_0 \subsetneq E_1 \subsetneq \cdots \subsetneq E_r = E,$$a filtration of subbundles. Defining $L_i = E_i /E_{i-1}$ for $i = 1,\dots,r$, the list $(L_1 , \dots , L_r )$ is called a \textbf{splitting} of $E$. A splitting $(L_1,\dots, L_r)$ of $E$ is said to be \textbf{maximal} if, for each $i$ in $1\leq i\leq r$, we have that $L_i$ is a line subbundle of maximal degree in the successive quotient of $E$ by the line bundles $L_0,\dots,L_{i-1}$ (in that order).
\end{definition}

To be clear about the successive quotient, we mean that $L_1$ must be maximal in $E/L_0\cong E$, that $L_2$ must be maximal in $((E/L_0))/L_1$, that $L_3$ must be maximal in $(((E/L_0)/L_1)/L2)$, and so on.

\begin{lemma}\label{maxl}
    \emph{(Lemma 11 \cite{Ati})} Let $E$ be an indecomposable vector bundle of rank $r$ and degree $r$ over an elliptic curve $X$. Then $E$ has a maximal splitting $(L,\dots ,L)$ with $\deg(L) = 1$.
\end{lemma}
Now consider an indecomposable vector bundle $E$ of rank $2$ and degree $2$. It admits a subbundle $L$ of rank $1$ with $\deg(L) = 1$ and $L \cong E/L$. We recall that the pushforward operation of bundles commutes with the quotients. Also we mention from our previous computation that $f_*L \cong \mathcal{O}\oplus \mathcal{O}(-1)$ while $L$ is a line bundle of degree $1$. Using these we have $\mathcal{O} \oplus \mathcal{O}(-1) \cong f_*E/(\mathcal{O} \oplus \mathcal{O}(-1))$. Thus, we can write the transition data of $f_*E$ directly as
\begin{equation}\label{tran}
g_{\alpha\beta}(z) = \begin{bmatrix}
  \begin{matrix}
  1 & 0 \\
  0 & \frac{1}{z}
  \end{matrix}
  & & h_{\alpha\beta}(z) \\

  \mathbf{0} & &
  \begin{matrix}
  1 & 0 \\
  0 & \frac{1}{z}
  \end{matrix}
\end{bmatrix}
\end{equation}
The function $h_{\alpha\beta}(z)$ uniquely corresponds to an element in $H^1(\mathbb{P}^1, \text{Hom}(E/F, F))$, which is$$H^1(\mathbb{P}^1, \text{End}(\mathcal{O}\oplus \mathcal{O}(-1)))$$in \ref{tran}. On the other hand, by applying aforementioned properties of the endomorphism bundle in combination with Serre duality, we obtain $\dim_{\mathbb{C}} H^1(\mathbb{P}^1, \text{End}(\mathcal{O}\oplus \mathcal{O}(-1))) = 0$. The zero element in the vector space uniquely corresponds to the splitting of $E$ as $F\oplus E/F$ (cf. Proposition 2 in \cite{Ati2}). Taking all of this together, we may write \ref{tran} as $$g_{\alpha\beta}(z) = \begin{bmatrix}
  \begin{matrix}
  1 & 0 \\
  0 & \frac{1}{z}
  \end{matrix}
  & & \mathbf{0} \\

  \mathbf{0} & &
  \begin{matrix}
  1 & 0 \\
  0 & \frac{1}{z}
  \end{matrix}
\end{bmatrix},$$and finally
\begin{equation}
f_*E\cong \mathcal{O} \oplus \mathcal{O} \oplus \mathcal{O}(-1) \oplus \mathcal{O}(-1).
\end{equation}
However, if $F$ is a subbundle of $E$ then $H^1(X, \text{Hom}(E/F, F))$ is not necessarily trivial. For example, one can choose $E \cong \mathcal{O}(2)  \oplus \mathcal{O}(4) \oplus \mathcal{O}(6) \oplus \mathcal{O}(6)$ and $F \cong \mathcal{O}(4) \oplus \mathcal{O}(6)$.\\

We end this section with a straightforward lemma:

\begin{lemma}\label{semi}
Let $X$ be an algebraic curve. If $M_1$ and $M_2$ are line bundles with same degree over $X$ then $M_1\oplus M_2$ is semistable.\end{lemma}

This result is true in arbitrary ranks, but we restrict ourselves to rank $1$.

\begin{proof} Let us consider a sub-line bundle $M$ of $M_1\oplus M_2$. Then, one of the following bundle morphisms must be nonzero: $M\xrightarrow[]{i} M_1\oplus M_2\xrightarrow[]{\pi_1} M_1;~M\xrightarrow[]{i} M_1\oplus M_2\xrightarrow[]{\pi_2} M_2$, where $\pi_1$ and $\pi_2$ are the bundle projection maps on $M_1$ and $M_2$. This leads to one of $H^0(X, \text{Hom}(M, M_1))$ or $H^0(X, \text{Hom}(M, M_2))$ being nontrivial. Thus, $\deg(M)\leq \deg(M_1) = \deg(M_2)$. Thus $\mu(M)\leq\mu(M_1\oplus M_2)$.
\end{proof}

Now, let us choose a smooth $\mathcal{O}(2)$-twisted spectral elliptic curve and its degree $2$ spectral covering map $\pi$ on $\mathbb{P}^1$. We denote the respective spectral polynomial by $p$. The spectral correspondence suggests that a rank $4$ semistable co-Higgs sheaf which $p$ annihilates must be the pushforward of a semistable rank $2$ bundle $M$ on the spectral curve. If $M$ is indecomposable, it is difficult to capture the matrix form of the co-Higgs pair $(\pi_*M, \pi_*\eta)$, although we have course already characterized the vector bundles which qualify as underlying bundles of such a pairs. We choose now a decomposable semistable bundle $M = L_1\oplus L_2$. It suffices to consider these two cases:\\ 

\begin{itemize}
\item $\deg(L_1) = \deg(L_2) = 0$; and\\
\item $\deg(L_1) = \deg(L_2) = 1$.\\
\end{itemize}

These degree choices lead directly to the following bundles $\pi_*(L_1 \oplus L_2)$:\\

\begin{itemize}
\item $\mathcal{O}\oplus\mathcal{O}(-1)\oplus \mathcal{O}(-1)\oplus\mathcal{O}(-2)$;\\
\item $\mathcal{O}\oplus\mathcal{O}\oplus \mathcal{O}(-2)\oplus\mathcal{O}(-2)$;\\
\item $\mathcal{O}(-1)\oplus\mathcal{O}(-1)\oplus \mathcal{O}(-1)\oplus\mathcal{O}(-1)$; and\\
\item $\mathcal{O}\oplus\mathcal{O}\oplus \mathcal{O}(-1)\oplus\mathcal{O}(-1)$.\\
\end{itemize}

In the light of the spectral correspondence, $\pi_*(L_1 \oplus L_2, \eta) = (E_1, \phi_1)\oplus (E_2, \phi_2)$ is semistable and $E_1$ and $E_2$ have the same degree. The characteristic polynomial of $(E_i, \phi_i)$ is $p$ and so $(E_i, \phi_i)$ is stable. Furthermore, each $(E_i, \phi_i)$ fits into a Jordan-H\"older filtration $0\subsetneq (E_i, \phi_i) \subsetneq (E, \phi)$.

\begin{remark}
In the above construction, a Jordan-H\"older series of co-Higgs sheaves is obtained immediately. The pairs $(E_i, \phi_i)$ are stable for $i = 1, 2$ and each of them induces a Jordan-H\"older co-Higgs subsheaf within their parent pair. \end{remark}

In any case, a moduli theoretic description of the nonabelian `fiber' is given by a result in \cite{Tu}:

\begin{theorem}
   \emph{(Theorem 1 \cite{Tu})} The moduli space $\mathcal{M}_{n,d}(C)$ of S-equivalence classes of semistable bundles of rank $n$, degree $d$ over an elliptic curve $C$ is isomorphic to the $h$-th symmetric product, $S^hC$, of the curve $C$, where $h=\gcd(n,d)$. 
\end{theorem}
\begin{remark}
    The moduli space of stable bundles of rank $n$ and degree $d$ is isomorphic to $C$ when $n$ and $d$ are relatively prime and empty otherwise.
\end{remark}
\section{A Composite projection formula}\label{Ite}

In the rest of the article, we illustrate another side of the classical spectral correspondence. In category theory, it is often asked if a morphism between two objects can be decomposed into intermediate morphisms between other objects. It is a perfectly natural question to ask if a multifold categorical correspondence can be established from the spectral viewpoint. We boil this question down to an investigation of the decomposition of nonconstant holomorphic maps in the category of compact Riemann surfaces. The morphisms here are spectral covering maps for such surfaces. Finally, a composite version of the holomorphic projection formula lifts our study to the category of vector bundles and twisted pairs. We focus primarily on the case of the projective line as the base. As a result of our prior arguments and computations over $\mathbb P^1$, we may embed the Jacobian of a generic cyclic spectral curve into the open subscheme of stable Hitchin pairs on an intermediate curve that is branched over $\mathbb P^1$ and which itself has, as a branched cover, the initial spectral curve. We first prove a composite version of the relevant ``push-pull' projection formula. Many of the relevant tools and perspectives that we rely upon follow from \cite{gun}. The following theorem is well known, but we include a straightforward proof for completeness.
\begin{theorem}\label{func}
Let $f: (X, \mathcal{O}_X) \to (Y, \mathcal{O}_Y)$ and $g: (Y, \mathcal{O}_Y) \to (Z, \mathcal{O}_Z)$ be morphisms of complex manifolds. Consider a sheaf $\mathcal{F}$ of $\mathcal{O}_X$-modules and a sheaf $\mathcal{G}$ of $\mathcal{O}_Z$-modules. Then, \emph{(i)} $(g\circ f)^*\mathcal{G} \cong f^*g^*\mathcal{G}$; and \emph{(ii)} $(g\circ f)_*\mathcal{F} \cong g_*f_*\mathcal{F}$.
\end{theorem}
\begin{proof} To prove (i) we use sheaf isomorphisms of $\mathcal{O}_X$-modules given as $(g \circ 
 f)^{-1}\mathcal{G} \cong f^{-1}g^{-1}\mathcal{G}$ and $$f^{-1}(g^{-1}\mathcal{G} \bigotimes\limits_{g^{-1}\mathcal{O}_Z} \mathcal{O}_Y) \cong f^{-1}g^{-1}\mathcal{G} \bigotimes\limits_{f^{-1}g^{-1}\mathcal{O}_Z} f^{-1}\mathcal{O}_Y.$$
Recall that a given morphism of sheaves is an isomorphism of sheaves if and only if it boils down to an isomorphism of the respective stalks at points. These sheaf isomorphisms can be verified by the isomorphisms of the stalks at points. At each point $y\in Y$
$(g^{-1}\mathcal{G})_y \cong \mathcal{G}_{g(y)}$. Thus $(f^{-1}g^{-1}\mathcal{G})_x \cong (g^{-1}\mathcal{G})_{f(x)} \cong \mathcal{G}_{g(f(x))}$ at each point $x\in X$. Keeping a bimodule structure of sheaves and stalks of sheaves in mind, we directly write the sheaf isomorphisms 
\begin{align*}
&f^*g^*\mathcal{G} = f^*(g^{-1}\mathcal{G} \bigotimes\limits_{g^{-1}\mathcal{O}_Z} \mathcal{O}_Y)\\
&= f^{-1}(g^{-1}\mathcal{G} \bigotimes\limits_{g^{-1}\mathcal{O}_Z} \mathcal{O}_Y)\bigotimes\limits_{f^{-1}\mathcal{O}_Z} \mathcal{O}_X \\
&\cong (f^{-1}g^{-1}\mathcal{G} \bigotimes\limits_{f^{-1}g^{-1}\mathcal{O}_Z} f^{-1}\mathcal{O}_Y)\bigotimes\limits_{f^{-1}\mathcal{O}_Y} \mathcal{O}_X\\
&\cong f^{-1}g^{-1}\mathcal{G} \bigotimes\limits_{f^{-1}g^{-1}\mathcal{O}_Z} (f^{-1}\mathcal{O}_Y\bigotimes\limits_{f^{-1}\mathcal{O}_Y} \mathcal{O}_X)\\
&\cong f^{-1}g^{-1}\mathcal{G} \bigotimes\limits_{f^{-1}g^{-1}\mathcal{O}_Z} \mathcal{O}_X\\
&\cong (g \circ f)^{-1}\mathcal{G} \bigotimes\limits_{(g\circ f)^{-1}\mathcal{O}_Z} \mathcal{O}_X = (g \circ f)^*\mathcal{G}.
\end{align*}
Proof of (ii) follows from $(g \circ f)_*\mathcal{F}(U) = g^{-1}f^{-1}(U) = g^{-1}(f_*\mathcal{F}(U)) = (g_*f_*\mathcal{F})(U)$.
\end{proof}

\begin{theorem}\label{push}
    If $f: (X, \mathcal{O}_X) \to (Y, \mathcal{O}_Y)$ is a morphism of ringed spaces and if $\mathcal{F}$ is a sheaf of $\mathcal{O}_X$-modules and $\mathcal{E}$ is a locally free sheaf of $\mathcal{O}_Y$-modules of finite rank then there is a natural isomorphism of sheaves of $\mathcal{O}_Y$-modules $f_*(\mathcal{F}\bigotimes\limits_{\mathcal{O}_X} f^*\mathcal{E}) \cong f_*\mathcal{F}\bigotimes\limits_{\mathcal{O}_Y} \mathcal{E}$.
\end{theorem}
The following corollary is immediate from Theorem \ref{push}.
\begin{corollary}\label{itpush}
Let $f: (X, \mathcal{O}_X) \to (Y, \mathcal{O}_Y);~ g: (Y, \mathcal{O}_Y) \to (Z, \mathcal{O}_Z)$ and $h: (X, \mathcal{O}_X)\to (Z, \mathcal{O}_Z)$ be morphisms of complex analytic spaces satisfying $h = g\circ f$. If $\mathcal{F}$ is a sheaf of  $\mathcal{O}_X$-modules and $\mathcal{E}$ is a locally free sheaf of $\mathcal{O}_Z$ modules of finite rank then there is an isomorphism of $\mathcal{O}_Y$-modules $f_*(\mathcal{F}\bigotimes\limits_{\mathcal{O}_X} h^*\mathcal{E}) \cong f_*\mathcal{F}\bigotimes\limits_{\mathcal{O}_Y} g^*\mathcal{E}$.
\end{corollary}
The isomorphism in Corollary \ref{itpush} is our \textit{composite projection formula}. The next remark is the key information that we will explore in the rest of this article.
\begin{remark}\label{midpush}
Under the assumptions in \ref{itpush}, the pushforward along $f$ of an $h^*\mathcal{E}$-twisted pair $(E, \phi)$ is a $g^*\mathcal{E}$-twisted pair on $Y$.
\end{remark}
In particular, we choose compact Riemann surfaces $X, Y, Z$ with nonconstant holomorphic maps $f: X\to Y,~ g: Y\to Z$ and $h:X\to Z$ satisfying $h = g\circ f$. We fix holomorphic vector bundles $F$ on $X$ and $E$ on $Z$. There is an isomorphism of vector bundles from Corollary \ref{itpush} directly $f_*(F \otimes h^*E)\cong f_*F \otimes g^*E$. On the other hand, denote by $\mathcal{O}(\mbox{GL}(r,\mathbb{C}))$ the (multiplicative) sheaf of holomorphic maps, on a compact Riemann surface, valued in nonsingular matrices of order $r$. Recall that a holomorphic vector bundle $E$ on $Z$ is an element $\{g_{\alpha\beta}\} \in H^1(Z, \mathcal{O}(\mbox{GL}(r,\mathbb{C})))$ where $r = \mbox{rank}(E)$. Thus, the pullback bundles $g^*E$ and $f^*g^*E$ have representatives $\{(g_{\alpha\beta}\circ g)\} \in H^1(Y, \mathcal{O}(\mbox{GL}(r,\mathbb{C})))$ and $\{(g_{\alpha\beta}\circ g\circ f)\} \in H^1(X, \mathcal{O}(\mbox{GL}(r,\mathbb{C})))$ respectively. We can, in the same spirit, define twisted bundle pairs on compact Riemann surfaces and interpret Remark \ref{midpush} in the context of holomorphic vector bundles: Let $L$ be a line bundle on $Z$ and $(E, \phi: E\to E\otimes h^*L)$ be a bundle pair on $X$. Then $(f_*E, f_*\phi: f_*E\to f_*E\otimes g^*L)$ appears as a ($g^*L$-twisted) bundle pair on $Y$. We organize this discussion in the following commutative diagram:

\vspace{5pt}

\begin{equation}\label{mmoc}
\begin{tikzcd}
 (E, h^*L, \phi) \to X \arrow[dr]{}[swap]{h} \arrow{rr}{f}[swap]{\scalebox{3.5}{$\circlearrowright$}} &&  Y\leftarrow (f_*E, g_*L, f_*\phi) \arrow{dl}{g} \\
                                & (h_*E, L, h_*\phi)\to Z
\end{tikzcd}
\end{equation}

So it is very much clear that the success in establishing a composite or a multifold (or an \textit{iterated}) spectral correspondence relies on the existence of a factorization of a spectral covering map.
\section{Imprimitive subgroups of a permutation group and monodromy groups}
The construction of the iterated spectral covers that we have addressed relates to a study of the Galois groups of the holomorphic covers. To serve this purpose, we briefly set the stage for a problem based on the modern aspect of computational group theory and enumerative algebraic geometry. The topological branched covers of the oriented $2$-manifolds constitute a research area overlapping with the computational theory of the permutation groups, graph theory, combinatorics, partition theory (\cite{Lan,perova}) and the theory of dessins d'enfants (\cite{wolf}). The classification of subgroups of the finite permutation groups into the \textit{primitive} and the \textit{imprimitive} categories opens up an angle to study the branched covers of $2$-surfaces. About 100 years ago, Ritt (\cite{ritt}) initiated techniques for studying the $2$-sphere, which lead to the modern study of the \textit{cartographic groups} or the \textit{monodromy groups} via the \textit{generalized constellations}. Ritt's theorem suggests that such groups are indispensable in illustrating factorization of the branched covers in to intermediate covers. A standard source for the details of the \textit{oriented hypermaps} and the cartographic groups is Chapter 1 \cite{Lan}. A stronger form of Ritt's theorem plays a decisive role in the Hurwitz problem as well (constructing branched maps between oriented surfaces with a prescribed branch data) (Lemma 5.2 \cite{perova} and Corollary 5.3 \cite{perova}). Interestingly, the cartographic groups are closely related to the Galois groups of holomorphic maps (cf. p.689 \cite{Joe}).\\

Let $X$ and $Z$ be compact connected Riemann surfaces. Given a nonconstant holomorphic map $\pi: X\to Z$ there is a unique degree $d$ such that fiber of each point $z\in Z$ contains $d$ points in $X$ counting up to multiplicities. There is a finite subset $B\subset Z$ such that $\pi: X' = X\backslash \pi^{-1}(B)\to Z\backslash B$ is a topological covering map of degree $d$ and is thus a local homeomorphism for each point $x\in X' =  X\backslash \pi^{-1}(B)$. Fix a point $z_0\in Z'= Z\backslash B$. Given a loop $\gamma$ based at $z_0$, lifts of $\gamma$ produce $d$ paths in $Z$ permuting points in the fiber of $z_0$. Collection of all such permutations forms a transitive subgroup of $S_d$ (recall that a subgroup $H$ of permutation group of $d$ symbols $S_d$  is said to be a \textit{transitive subgroup} if for each pair of symbols $a_i, a_j$ there is an element $\sigma\in H$ such that $\sigma(a_i) = a_j$) which can be realized as the image of a group homomorphism (cited as the \textit{monodromy representation}) $\rho: \pi_1(Z', z_0)\to S_d$. This transitive subgroup is defined to be the \textit{cartographic group} or the \textit{monodromy group} of $\pi$ at $z_0$. It is important to note that a group homomorphism $\rho: \pi_1(Z', z_0)\to S_d$ with a transitive image gives a compact Riemann surface $X$ with a nonconstant holomorphic map $\pi$ which has branch points in $B$ (p.91 \cite{Ric}).\\

The problem of computing the monodromy groups is, of course, a complicated one in general. It necessitates the application of several numerical techniques simultaneously. One can see this at play in, for example, sections 2.2, 3.1, and 3.2 of \cite{duff}. A closely related problem is the determination of the Galois group of a branched holomorphic cover. A nonconstant branched holomorphic map $\pi: X\to Z$ is not necessarily a Galois cover in the sense that the field extension $\mathcal{M}(X)/\pi^*\mathcal{M}(Z)$ is not Galois. We consider the Galois closure $\mathcal{M}(X)^{\mbox{Gal}}$ of $\mathcal{M}(X)$ and the Galois group $\mathcal{M}(X)^\text{Gal}/\pi^*\mathcal{M}(Z)$. Now, a theorem by Harris \cite{Joe} states that fixing a generic base point $z_0\in Z'$ we can embed the group $\mbox{Gal}(\mathcal{M}(X)^{\mbox{Gal}}/\pi^*\mathcal{M}(Z))$ in $S_d$ and the image of this embedding is same as the monodromy group. This result opens up a direction of computation for monodromy groups at a generic point --- particularly when we can determine the Galois groups from extensions defined by certain algebraic equations. We refer the reader to Chapter 1 of \cite{Lan} as well as to \cite{Alg} for some background here. 

\begin{definition}
Let $r\in \mathbb{N}$. A group of permutations $G$ over $r$ symbols $\{a_1,.., a_r\}$ is said to be \textbf{imprimitive} if there is a partition $\{B_1,\dots , B_k\}$ of $\{a_1,.., a_r\}$, each of size $l$ with $r >l > 1$, such that for each element $g\in G$ we have $g(B_i) = B_j$; that is, if the image of each block is a block again.
\end{definition}

\begin{example}\label{cyclic1}
    Any cyclic group generated by an $r$-cycle in the permutation group of $r$ symbols is transitive and in case $r$ is composite, it is an imprimitive subgroup. 
\end{example}
\begin{proof} Let $\{a_1,\dots , a_r\}$ be $r$ distinct symbols. It suffices to prove the statement only for $<\sigma = (a_1\dots a_r)>$. Choose $i, j$, then $\sigma^{i - 1}(a_1) = a_i$ and $\sigma^{j - 1}(a_1) = a_j$. Thus $\sigma^{j-1}\circ~\sigma^{1-i}(a_i) = a_j$. Let $u\geq 2$ be a divisor of $r$ and $r = uv$. Organize $v$ blocks namely $B_1 = [a_1, a_{u + 1},\dots , a_{(v - 1)u + 1}],~ B_2 = [a_2, a_{u + 2},\dots , a_{(v - 1)u + 2}],\hdots,~B_u = [a_p, a_{2p},\dots , a_r]$. It suffices to prove that for each $B_i$, $\sigma(B_i) = B_j$ for some $j$ as $\sigma^2(B_i) = \sigma(B_j) = B_k$ and so on. Finally, the last statement is a direct consequence of $\sigma(B_i) = B_{i+1}$ for $i < u$ and $\sigma(B_u) = B_1$. This finishes our argument.
\end{proof}

\begin{example}\label{cyclic2}
    Let $r$ be a composite number. Consider a subgroup of order $r$ of $S_r$ generated by an element $\sigma = \sigma_1\sigma_2$ where $\sigma_1$ and $\sigma_2$ are mutually disjoint cycles. This subgroup is imprimitive. 
\end{example}
\begin{proof} It suffices to find blocks of the same length so that the length divides $r$. Let the lengths of $\sigma_1$ and $\sigma_2$ be $l_1$ and $l_2$ respectively. So, $l_1, l_2 \geq 2$. Let $q = \gcd(l_1, l_2)$ and suppose that $q>1$ (that is, that $l_1$ and $l_2$ are \emph{not} coprime). We make blocks of length $q$ out of $\sigma_1$ and $\sigma_2$ as per the procedure laid out in Example \ref{cyclic1}. Among the remaining $r - (l_1 + l_2)$ elements we pick up identity blocks of length $q$ because $q$ divides $r - (l_1 + l_2)$. Otherwise, let $q = 1$. We assume without loss of generality that $l_1 > l_2$ because $l_1 = l_2$ is not an option. Also we have $l_1.l_2 = \mbox{lcm}\{l_1, l_2\} = r$. As $l_2\geq 2$ we have $r = l_1.l_2\geq 2l_1$. Among the remaining $r - (l_1 + l_2)$ elements choose $l_1 - l_2$ elements (they are mapped to themselves by $\sigma$) and attach to $\sigma_2$ to make a block $\sigma_2'$ of length $l_1$. Thus $\sigma_1$ gives a block of length $l_1$ say $B_{\sigma_1}$, $\sigma_2'$ gives a block of length $l_1$, say $B_{\sigma_2}$ and among the remaining $r - 2l_1$ elements, each mapping to itself by $\sigma$, we choose $\frac{r - 2l_1}{l_1}$ many blocks each of length $l_1$, denoting each of them as $B_i$. Now the subgroup generated by $\sigma$ is imprimitive because the action of $\sigma$ on $S_r$ preserves the blocks due to $\sigma(B_{\sigma_1}) = B_{\sigma_1}$ and $\sigma(B_{\sigma_2'}) = B_{\sigma_2'}$, and so finally $\sigma(B_i) = B_i$.
\end{proof}

\begin{remark}\label{cycl}
If a cyclic subgroup $H$ of order $r$ in $S_r$ is transitive, then it is generated by an $r$-cycle. 
\end{remark}
\begin{proof} Suppose that $H$ is generated by an element $\sigma=\sigma_1\cdots\sigma_k$ written in terms of disjoint cycles $\sigma_1,\dots,\sigma_k$. To begin, suppose that $k\geq2$. Let $b_1$ and $b_2$ be distinct symbols such that $b_1$ appears in $\sigma_1$ and $b_2$ appears in $\sigma_2$. Since $H$ is transitive, there exists an integer $i$ such that $\sigma^i$ maps $b_1$ to $b_2$. Since $\sigma^i = \sigma_1^i\cdots\sigma_k^i$, we observe that $\sigma^i(b_1)$ is a symbol that appears in $\sigma_1^i$ while but $b_2$ does not appear in $\sigma_1^i$. This is contradictory, and so we must have $k=1$. Thus, $\sigma$ is an $r$-cycle, from which the remark now follows.\end{proof}

As an example of a transitive subgroup, we can present the monodromy or cartographic group of an $r : 1$ branched cover. The monodromy group acts transitively on the fiber of an unramified point. If the 3monodromy group is a cyclic group of order $r$ then it is an imprimitive subgroup (cf. Example \ref{cyclic1}).

\begin{definition}
    A map $\pi: Y\to X$ between compact Riemann surfaces is said to be \textbf{factorizable} if there exists a compact Riemann surface $Z$ and nonconstant holomorphic maps $f: Y\to Z$ and $g: Z\to X$, both of degree $> 1$ such that $\pi = g\circ f$.
\end{definition}

For completeness, we develop an abridged yet nonetheless explicit proof of Ritt's theorem that states that a branched cover of surfaces decomposes into two intermediate branched covers precisely when the monodromy or Galois group is imprimitive. A topological argument for Ritt's theorem for connected oriented $2$-manifolds, as in in Theorem 1.7.6 of \cite{Lan}, is achieved via a purposeful engineering of monodromy groups. A similar proof can be achieved for compact Riemann surfaces. Keeping this in mind, we highlight a prominent example of the application of Ritt's theorem:
\begin{corollary}
    If the monodromy/Galois group of an $r:1$ holomorphic cover $\pi$ of Riemann surfaces is a cyclic group of order $r$, then $\pi$ is factorizable if and only if $r$ is composite. 
\end{corollary}
This corollary is a statement included implicitly within Proposition 2.17 of \cite{duff}, and one can infer it from Example \ref{cyclic1} and Remark \ref{cycl}. However, we prefer an alternative argument involving function fields over $\mathbb{C}$. The result \ref{Intermediate} that we develop in the following section is possibly already known; however, we are not aware of a specific reference including a proof.

\section{Factorization through Galois groups of covers}
A map $\pi:Y\to X$ of Riemann surfaces gives a set theoretic inclusion $\pi^*\mathcal{M}(X)\subset\mathcal{M}(Y)$ where $\pi^*\mathcal{M}(X)$ denotes the field of meromorphic functions on $Y$ of the form $\pi\circ f$ with $f\in\mathcal{M}(X)$. If $\pi$ is factorizable there exists an intermediate subfield $\pi^*\mathcal{M}(X)\subset K\subset \mathcal{M}(Y)$. The converse statement is also true: to each intermediate subfield of this field extension, we may assign an intermediate compact Riemann surface and a pair of branched holomorphic covering maps whose composition recovers $\pi$. We will provide a rigorous construction of such an intermediate Riemann surface. The techniques of the construction will connect the three categories-- the compact connected Riemann surfaces, the function fields over $\mathbb{C}$ and, the complex smooth irreducible projective algebraic curves.  
\begin{proposition}\label{Intermediate}
    A nonconstant branched holomorphic map $\pi:Y \to X$ of compact connected Riemann surfaces is factorizable if and only if there exists a proper intermediate subfield $\pi^*\mathcal{M}(X)\subset K\subset \mathcal{M}(Y)$. 
 \end{proposition}
 Let $\pi: Y\to X$ be a nonconstant holomorphic map of finite degree $n$. Then $\pi^*\mathcal{M}(X)\subset\mathcal{M}(Y)$ is a finite field extension of degree $n$. Let $\pi^*\mathcal{M}(X)\subseteq E\subseteq\mathcal{M}(Y)$ is an intermediate field, we will show there exist a compact Riemann surface $Z$, finite covering maps $f: Y\to Z$ and $g: Z\to X$ such that $g\circ f = \pi$ and $f^*\mathcal{M}(Z) = E$. We use following results from p.64 in \cite{Giro}:

 \begin{theorem}
        Let $X_1$ and $X_2$ be compact Riemann surfaces and $\Sigma_1, \Sigma_2$ be finite subsets of $X_1, X_2$ respectively. Assume that $X_1^* = X_1\backslash \Sigma_1$ and $X_2^* = X_2\backslash \Sigma_2$ are isomorphic. Then $X_1$ and $X_2$ are isomorphic.
    \end{theorem}

    \begin{theorem}
        Let $Y$ be a compact Riemann surface, $\Sigma\subset Y$ be a finite set. $f^*:X^*\to Y^*$ is an unramified holomorphic covering of finite degree. Then there exists a unique compact Riemann surface $(X^*\subset) X$ such that $f^*$ extends a unique morphism $f: X\to Y$. Moreover $X\backslash X^*$ is a finite set.
    \end{theorem}

    \begin{corollary}\label{extens}
        Let $X, Y$ be compact Riemann surfaces and $\Sigma_1\subset X, \Sigma_2\subset Y$ be finite subsets. An unramified holomorphic covering of a finite degree $f^*: X\backslash \Sigma_1\to Y\backslash \Sigma_2$ extends to a morphism (i.e. a nonconstant holomorphic map) $f:X\to Y$.
    \end{corollary}
We prove Proposition \ref{Intermediate} formally in the case of $X = \mathbb{P}^1$ as we can rely directly upon the existence of polynomial equations in two variables.

\begin{proof} Let $\pi: Y\to \mathbb{P}^1$ have degree $n \geq 1$. Let us consider, as in the statement of the proposition, a proper intermediate field $E$. Then, $E = \pi^*\mathcal{M}(\mathbb{P}^1)(\alpha)$, and $\{1, \alpha,\dots , \alpha^{r-1}\}$ is a basis of $E$ over $\pi^*\mathcal{M}(\mathbb{P}^1)$. As such, as a field extension, $E$ is of degree $r$ with $1 \leq r \leq n$ and $E = \mathbb{C}(\pi, \alpha)$. We produce an irreducible polynomial $F(x, y) \in \mathbb{C}[x, y]$ such that $F(\pi, \alpha) = 0$ and a compact connected Riemann surface $X^F$ compactifying the zero locus of $F$. Moreover, $\mathcal{M}(X^F) = \mathbb{C}(\mathbf{x}, \mathbf{y})$ where $\mathbf{x}, \mathbf{y}$ denote the holomorphic projection of coordinates.

A standard construction of $F$ is available at p.68 in \cite{Giro} and pp.22--24 in \cite{wolf}. Let the irreducible minimal polynomial of $\alpha$ over $\mathcal{M}(\mathbb{P}^1)$ be$$M(T) = T^r + \pi^*a_1 T^{r - 1} +\dots + \pi^*a_r.$$Let $\psi$ denote the polynomial over $\pi^*M(\mathbb P^1)$ whose coefficients are the symmetric functions of $\alpha$. As $\alpha$ is annihilated by $\psi$, the minimal polynomial divides $\psi$ over $\pi^*\mathcal{M}(\mathbb{P}^1)$. We clear the denominators of $a_1,\dots , a_r$ by multiplying by their least common multiple to obtain a complex irreducible polynomial $F(x, y)$ out of $M$.

Now, let $\phi:Y\to \mathbb{P}^1\times\mathbb{P}^1$ the holomorphic map given by $\phi(y) = (\pi(y), \alpha(y))$. We restrict the map to image $\phi(Y)$, a compact connected analytic variety. Moreover $\phi$ is proper --- that is, preimage of a compact set is compact. Writing$$F(x, y) = p_0(x)y^r + p_1(x)y^{r-1}+\dots +p_r(x)$$with $p_0(x)$ a nonzero polynomial in $x$, we define a connected smooth Riemann surface $$C^x := \{(x_0, y_0)\in Z(F): F_x(x_0, y_0)\neq 0; p_0(x_0)\neq 0\}$$ and a holomorphic projection map $\mathbf{x}: C^x\to \mathbb{P}^1$ (cf. p.69 in \cite{Giro}). The charts of $C^x$ are furnished by the Implicit Function Theorem. We complete $C^x$ to $X^F$ by adding finitely-many points. We remove a finite set of points which are \textit{not} valued in $\mathbb{C}^2$ so that the image lies inside $Z(F)$. Moreover, the irreducible polynomial $F$ intersects $F_x$ at finitely many points and the common solutions of $p_0$ and $F$ are finite. Thus, we remove finitely many points from $Z(F)$ and fibers of these points (under $\phi$) and define a map $\phi: Y''\to C^x$. This map is proper and holomorphic. So, it has a finite degree. Finally, we see that $\phi$ has finitely many ramification points in $Y''$ because $\pi$ has finitely many ramification points in $Y$. We remove finitely many branch points from $C^x$ and ramification points from $Y''$ to get a unbranched map $\phi:Y'\to C^{x'}$. The restricted map of $\phi$ rewritten $f': Y'\to X'$ is an unramified holomorphic covering of finite degree. In combination with Corollary \ref{extens}, we are led to the existence of a holomorphic map$$f:Y\to X^F.$$From the fact that $\mathcal{M}(X^F) = \mathbb{C}(\mathbf{x},\mathbf{y})$ (cf. p.74, Corollary 1.93 \cite{Giro}) we obtain that $f^*\mathbf{x} = \pi$ and $f^*\mathbf{y} = \alpha$, for which we invoked the identity theorem. Thus $E = f^*\mathcal{M}(X^F)$. The map $g: X^F\to \mathbb{P}^1$ is the map $\mathbf{x}$. We obtain that $g(f(y)) = \pi(y)$ on $Y$ punctured at finitely many points; thus, $\pi = g\circ f$ is indeed true.
\end{proof}

\begin{remark}
As side note, we mention that $[E:\pi^*\mathcal{M}(\mathbb{P}^1)] = \deg(g) = r$ and $[\mathcal{M}(Y): E] =\deg(f) = n/r = m$.
\end{remark}

We generalize the above argument to curves of higher genus. We simply choose a compact Riemann surface defined by an irreducible polynomial over the function field of $X$. This is equivalent to analytically continuing germs of local solutions of an algebraic equation. The reader may compare to p.53, Theorem 8.9 in \cite{bre1}.\\ 

We include, for completeness, standard theorem that explains a construction of a compact Riemann surface out of an algebraic equation. This theorem confirms that our construction of an intermediate cover for the genus-$0$ case can be generalized to any genus.
\begin{theorem}\label{polsol}
    Suppose that $X$ is a compact Riemann surface and that$$P(T) = T^r + c_1T^{r - 1} +\dots + c_r\in \mathcal{M}(X)[T]$$is an irreducible polynomial of degree $r$. Then, there exists a compact Riemann surface $Y$ that can be realized as a branched holomorphic $r:1$ cover $\pi: Y\to X$ as well as a meromorphic function $F\in\mathcal{M}(Y)$ satisfying $(\pi^*P)(F) = 0$. The triple $(Y, \pi, F)$ is determined uniquely: if $(Z, \tau, G)$ satisfies the same properties, then there exists exactly one biholomorphic mapping $\zeta: Z\to Y$ such that $G = \zeta^*F$.
\end{theorem}

\begin{theorem}\label{Ritt}
    (Ritt) A nonconstant $r:1$ holomorphic map between compact Riemann surfaces $\pi: Y\to X$ is factorizable if and only if the Galois group of the branched covering $\emph{Gal}(\mathcal{M}(Y)^\emph{Gal}/\pi^*\mathcal{M}(X))$ is imprimitive.
\end{theorem}
\begin{proof} We have already referred to a topological proof using exclusively monodromy groups (p.65, Theorem 1.7.6 \cite{Lan}) that involves the construction of $2$-surfaces out of generalized constellations. However, we outline a proof available in the spirit of the that of Proposition 1 of \cite{Alg}. One direction can be tackled as follows: we identify the Galois group with the cartographic group at a generic point through an argument of Harris as per p.689 of \cite{Joe}, then prove that the cartographic group is imprimitive. To prove the converse, we return to the actual Galois group itself. Let $K$ be the Galois closure of the field extension $\mathcal{M}(Y)/\pi^*\mathcal{M}(X)$. Our aim is to set up a proper subfield between $\mathcal{M}(Y)$ and $\pi^*\mathcal{M}(X)$. Let us recall the Fundamental Theorem of Galois Theory briefly.

Let $K/L$ be a finite Galois extension of fields. Then there is an inclusion-reversing bijective correspondence between (i) the fixed subfields $K^H$ intermediate between $K/L$ corresponding to a subgroup $H$ of $\text{Gal}(K/L)$ and (ii) the automorphism groups $\text{Aut}(J/L)$ for an intermediate field $J$ between $K/L$. Finally, the degree of extension, $[K^H : L]$ is same as the group index $[\text{Gal}(K/L): H]$. So, the subgroups of the Galois group $G = \text{Gal}(K/\pi^*\mathcal{M}(X))$ are in one-to-one (inclusion reversing) correspondence with the intermediate subfields. We choose a nontrivial block $B_1$. Its set of stabilizers is a subgroup $H'$ of $G$.  The group $G$ contains a subgroup $H$ whose fixed field $K^H$ is same as $\mathcal M(Y)$. Then $H'$ is a proper subgroup of $G$ properly containing $H$. We consider a compact Riemann surface $\tilde{X}$ associated to $K^{H'}$. We can come up with the intermediate covering maps of $\pi$ defined by the inclusion $\pi^*\mathcal{M}(X)\subset K^{H'}\subset\mathcal{M}(Y)$ by Proposition \ref{Intermediate}.
\end{proof}

According to Theorem \ref{Ritt}, the existence of a decomposition of a generic spectral covering into two intermediate maps depends entirely on the Galois group. We are in a position now to pose a question that, to our knowledge, remains unexplored. The question is appealing as it bridges pure aspects of geometry with computational group theory. Let $X$ be a smooth irreducible projective algebraic curve. Given a holomorphic line bundle $L$ over $X$, we recall that smoothness of a spectral curve defined by $s = (s_1,\dots , s_r)\in \bigoplus_{i = 1}^r H^0(X, L^i)$ is an open condition. That is, given a suitable $s_0$, we have smooth spectral curves $X_s$ for each $s$ near $s_0$. An immediate question is: are the Galois groups invariant under small perturbations of spectral coefficients? (In asking this, we should note that a spectral covering map is not necessarily Galois, and we ought to consider Galois closures of field extensions.) Whether Galois groups remain imprimitive for any generic $s$ is an open question. It is worth probing how this scenario plays out over $\mathbb{P}^1$. A set of spectral coefficients defines a unique set of meromorphic functions on $\mathbb{P}^1$ and, in turn, a compact connected Riemann surface that covers $\mathbb{P}^1$. If the underlying spectral Riemann surface is smooth and connected, then there is a fiber preserving isomorphism between these two surfaces. In particular, we have that a generic cyclic spectral curve is isomorphic to a covering Riemann surface and the spectral covering map is equivalent to a covering map that admits a cyclic Galois group. For a compact Riemann surface of higher genus, the following preliminary conjecture may be formulated: 

\footnote{We are aware through some discussions in the community that the conjecture holds for higher genus as well.  That said, we are unaware of any definite reference with a proof.}\label{fn}
\begin{conjecture}\label{conj}
Let $X$ be a compact Riemann surface of genus $g_X$; $L\to X$, a holomorphic line bundle on $X$ of positive degree. We fix a generic section $s\in H^0(X, L^r)$ with distinct zeros. Consider the extension of function fields defined by the smooth integral cyclic spectral cover $\pi: X_s\to X$. This extension is Galois and the Galois group of cover $\pi$ is a cyclic group of order $r$.\footref{fn}
\end{conjecture}

We verify this conjecture for $g_X = 0$. From the classical spectral correspondence \cite{BNR} we can embed the Jacobian of a spectral curve inside a quasi-projective variety of the S-equivalent semistable pairs over an intermediate spectral curve.

\section{Galois theory of cyclic spectral covers of $\mathbb P^1$}\label{compl}

Here, we pin down the Galois theory of generic cyclic spectral covers. The sections $s_i$ are complex polynomials over the affine coordinate charts of $\mathbb{P}^1$ which we can assume as meromorphic functions on $\mathbb{P}^1$. This understanding leads to the construction of an irreducible algebraic polynomial equation over the function field of $\mathbb{P}^1$. Theorem \ref{polsol} addresses an analytic approach to prove that there is a root of such an equation in a finite extension $K$ of $\mathcal{M}(\mathbb{P}^1)$ and there is a unique compact Riemann surface $Y$ such that $K$ is $\mathbb{C}$-algebra isomorphic to $\mathcal{M}(Y)$.\\

Our approach borrows from the construction of hyperelliptic curves (cf. \cite{Ric}). Let $t\geq 2$ and $r\geq 2$. Choose a generic section $s$ of $\mathcal{O}(tr)$ which has $tr$ distinct zeros $B = \{z_1,\dots ,z_{tr}\}\subset \mathbb{P}^1$. It can be represented by a complex polynomial $s$ with distinct zeros of degree $tr$ or $tr - 1$. Making a small change in notation, the spectral curve $X_s$ corresponding to $\lambda^r - s$ is given by 
\begin{equation}\label{Spec5}
\begin{cases}
y^r - s(x) = 0\\
\tilde{y}^r - \tilde{s}(\tilde{x}) = 0
\end{cases}
\end{equation}
with identification given in equations \ref{spec}. Strictly speaking, this following construction only makes sense in case section $s$ has distinct zeros. If we choose $s$ with repeated zeros we obtain a singular spectral curve. These singular curves admit singularity at the multiple zeros of $s$ and $\tilde{s}$.  We are able to desingularize a singular spectral curve with multiple techniques but there is no guarantee that the spectral correspondence will hold for that desingularized spectral curve.\\

The zeros of the section $s$ are the only branch points of the spectral covering map, each having a singleton fiber. There are exactly $tr$ ramification points, each with multiplicity $r$ in $X_s$. We recall the definition of the polynomial $\tilde{s}(z) = z^{tr}s(\frac{1}{z})$, and we note that both $s$ and $\tilde{s}$ admit simple roots and are of degree equal to either $tr$ or $tr - 1$. We have already presented the associated affine plane curves in $\mathbb{C}^2$ as $$C_1 = \{(x, y)\in \mathbb{C}^2: y^r = s(x)\}; C_2 = \{(\tilde{x}, \tilde{y})\in \mathbb{C}^2: \tilde{y}^r = \tilde{s}(\tilde{x})\}.$$ Here $C_1, C_2$ are noncompact smooth connected Riemann surfaces due to the fact that $s$ and $\tilde{s}$ admit distinct roots. To establish that $X_s$ is a compact Riemann surface, consider the following open subsets $U$ and $V$ of $X$ and $Y$ respectively:
$$U := \{(x, y)\in \mathbb{C}^2: y^r = s(x); x\neq 0\}; V := \{(\tilde{x}, \tilde{y})\in \mathbb{C}^2: \tilde{y}^r = \tilde{s}(\tilde{x}); \tilde{x}\neq 0\}.$$

We consider a holomorphic map $\psi: U\to V$ by $$\psi(x, y) = \left(\frac{1}{x}, \frac{y}{x^t}\right).$$ It is apparent that $\phi$ is an isomorphism. There are only finitely-many points in $C_1\backslash U$ and $C_2\backslash V$. As per the definition of the spectral curve $X_s$ we take disjoint union of $C_1 \sqcup C_2$ along $\psi$. That is, we identify each point in $C_1\backslash U$ to itself, each point in $C_2\backslash V$ to itself and each point $u\in U$ to itself or to $\psi(u)$. The `disjoint union' topology of $C_1\sqcup C_2$ descends to quotient topology of $C_1\sqcup C_2/\psi$. Finally, the space $ C_1\sqcup C_2/\psi$ which is nothing but $X_s$ is a compact (restriction on closed unit discs) connected (due to non-empty intersection of connected components) Hausdorff second countable topological space. The holomorphic charts of $C_1$ and $C_2$ produce charts of points in $X_s$ via inclusion maps on $C_1$ and $C_2$ (cf. p.60 in \cite{Ric}). Thus $X_s$ is a compact Riemann surface. Observe that we can embed $C_1$ and $C_2$ into $X_s$ and $X_s\backslash C_1$ and $X_s\backslash C_2$ are finite sets. Indeed, $X_s$ is compact completion of both $C_1$ and $C_2$. We want to find the genus of $X_s$ in an alternative way. To do this computation we want to obtain $X_s$ as a finite branched cover of $\mathbb{P}^1$. Observe that it is enough to understand the calculus over $C_1$ because $C_2$ contributes only finitely many points to $X_s$.\\

We have the first holomorphic projection coordinate map $\pi': C_1\to \mathbb{C}$ as a holomorphic surjective finite branched map with the zeros of $s$ as the branch points. (We can explore $\mathbb{P}^1$ as $\mathbb{C}\sqcup \mathbb{C}/\psi'$ while $\psi':\mathbb{C}^*\to \mathbb{C}^*$ is defined as $\psi'(z) = \frac{1}{z}$. Here $0\in\mathbb{P}^1$ is denoted by $0$ in first summand $\mathbb{C}$ and $\infty$ is denoted by second summand $\mathbb{C}$.) This is a restriction of the bundle map $\mathcal{O}(t)\to\mathbb{P}^1$. The pre-image of each branch point is singleton i.e. each of the ramification points of $\pi'$ has multiplicity $r$. The map is extended to a holomorphic branched covering map $\pi$ to $\mathbb{P}^1$ of finite degree $r$. In case $s$ has degree $tr$ we observe that $0$ is not a complex root of $\tilde{s}$ i.e. the zeros of $s$ are the only branch points of $\pi$. It is the same thing as saying that $\infty\in\mathbb{P}^1$ is not a branch point of $\pi$. On the other hand, in case $s$ has degree $tr - 1$, we have $0$ as a root of $\tilde{s}$ and there is a ramification point in $X_s$ which has multiplicity $r$ and $\pi$ maps that point to $\infty\in\mathbb{P}^1$. Thus in each case we have $tr$ ramification points in $X_s$ each with multiplicity $r$. We are in position to apply Hurwitz formula: $2(g_{X_s} - 1) = b_\pi + 2r(g_{\mathbb{P}^1} - 1) = tr(r - 1) - 2r \implies g_{X_s} = \frac{(tr-2)(r-1)}{2}$.
\begin{remark}
This calculation is consistent with the genus produced by Equation \ref{genus}.
\end{remark}

We have access to another meromorphic function on $X_s$, which we call $f_y$, given by projecting the $y$-coordinate. The $y$-coordinate map $C_1\to \mathbb{C}$ is extended to $C$ by mapping the points $(0, \tilde{y})$, which are contributed exclusively by the other affine component of the spectral curve, to $\infty\in \mathbb{P}^1$. The polynomial $s$ is a meromorphic function on $\mathbb{P}^1$ defining $\infty\mapsto\infty$. From the description of the holomorphic map $\pi$, we realize that $f_y^r = \pi^*s$ on $C_1$, which extends to an agreement over the whole of $X_s$. The (net) degree of $s$ as a meromorphic function on $\mathbb{P}^1$ is one of $tr$ or $tr - 1$, and $\pi^*s$ has degree $r\deg(s)$ and $\deg(f_y) = \deg(s)$. 

\begin{remark}
    In Theorem \ref{polsol}, setting $c_1 = \dots  = c_{r - 1} = 0$ and $c_r = -s$ yields $(X_s, \pi, f_y)$, which is the unique solution of the corresponding irreducible polynomial.
\end{remark}

With this, we have a complete understanding of the function field of $X_s$. We claim that $\mathcal{M}(X_s) = \mathbb{C}(\pi, f_y)$ or $\pi^*\mathcal{M}(\mathbb{P}^1)(f_y)$. First, we see that $\pi^*\mathcal{M}(\mathbb{P}^1)\subset  \pi^*\mathcal{M}(\mathbb{P}^1)(f_y)\subseteq\mathcal{M}(X_s)$ and both $\pi^*\mathcal{M}(\mathbb{P}^1)\subset \mathcal{M}(X_s)$ and $\pi^*\mathcal{M}(\mathbb{P}^1)\subset \pi^*\mathcal{M}(\mathbb{P}^1)(f_y)$ are $r:1$ extensions. We thereby reach the desired conclusion: that the field of meromorphic functions is$$\mathcal{M}(X_s) = \Bigl\{\sum_{j = 0}^{r-1}\pi^*r_jf_y^j: r_j\in \mathcal{M}(\mathbb{P}^1)\Bigr\}.$$We also obtain that $T^r - \pi^*s$ is a separable polynomial; that is, all of its roots are distinct, given by $\{\xi^if_y: i = 0,\dots , r-1\}$, wherein $\xi$ is an imprimitive $r$-th root of $1$. Indeed, $\mathcal{M}(X_s)$ is the splitting field of $T^r - \pi^*s$. Therefore, this extension is a Galois extension, as $\xi^if_y$ lies in $\mathcal{M}(X_s)$ for all $i$.\\

We are now in a position to compute the Galois group of the extension $\pi^*\mathcal{M}(\mathbb{P}^1)\subset \mathcal{M}(X_s)$. The equation $T^r - \pi^*s = 0$ over $\pi^*\mathcal{M}(\mathbb{P}^1)$ admits all roots in the function field of $X_s$. We observe that $f_y$ is an $r$-th root of $\pi^*s$, and we denote $f_y$ by $\sqrt[r]{\pi^*s}$ accordingly. The complete set of roots is $\{\sqrt[r]{\pi^*s}, \xi\sqrt[r]{\pi^*s},\dots , \xi^{r - 1}\sqrt[r]{\pi^*s}\}$. The action of an element of the Galois group maps $f_y$ to another root of this set, which thereby determines the action on the rest of the roots. The Galois group is cyclic and its generator $\sigma$ satisfies $\sigma(\sqrt[r]{\pi^*s}) = \xi \sqrt[r]{\pi^*s}$ in which $\xi$ denotes a primitive $r$-th root of unity $e^{\frac{2\pi i}{r}}$. On the other hand, $\sigma(\sqrt[r]{\pi^*s})$ is a root of the equation $T^r - \pi^*s = 0$.\\

Note that the deck transformation group (cf. p.57, Theorem 8.12 in \cite{bre1}) will be isomorphic to the Galois group. We may adopt the following approach of p.74 of \cite{Ric} (p.74) in computing the group of deck transformations: first, observe that $$(x,y)\mapsto (x, \xi^i y)$$ is a deck transformation on $X_s$ for $i = 0,\dots , r-1$. Indeed, these are the only deck transformations. Furthermore, for any $\sigma\in \text{Deck}(\pi)$, we have $\sigma^r(y) = s(x) = y^r$. For each point, we therefore have $\sigma^r(y) = \xi^i y$ for some $i$. By the continuity of the automorphism $\sigma$, we have that $i$ will be the same everywhere. Hence, we may identify the group of all deck transformations with the cyclic one generated by $\sigma_1 = (x,y)\mapsto (x, \xi y)$. We remark that the spectral cover defined by \ref{spec} nicely reflects the moniker \textit{cyclic cover} as used in p.73 of \cite{Ric}.\\

Let $r\geq 4$ be a composite number. We denote the corresponding Galois group by $G$ and use $H$ to denote a subgroup. We see that for each divisor $m$ of $r$ there is a unique subgroup of order $m$. The index of such a subgroup is $p$ where $r = m.p$. Now a subgroup $0\subset H\subset G$ of order $m$ uniquely associates to a finite extension $\mathcal{M}(X_s)\supset K\supset \pi^*\mathcal{M}(\mathbb{P}^1)$ by the Fundamental Theorem of Galois theory. From Proposition \ref{Intermediate} we obtain that there exists a holomorphic map $f: X_s\to X$ such that $K = f^*\mathcal{M}(X)$. Thus $\deg(f) = [\mathcal{M}(X_s) : K] = m$ and $\pi^*\mathcal{M}(\mathbb{P}^1)\subset K = f^*\mathcal{M}(X)$ such that there is a holomorphic map $g: X\to\mathbb{P}^1$ such that $\pi = g\circ f$.\\

We observe that the Riemann surface $X$ is uniquely determined up to isomorphism whenever we fix the degrees of the intermediate covering maps. Let us consider another pair of maps $\tilde{f}: X_s\to \tilde{X}$ and $\tilde{g}: \tilde{X}\to\mathbb{P}^1$ such that $\pi = \tilde{g}\circ \tilde{f}$ and $\deg(f) = \deg(\tilde{f}) = m$. It follows that $f^*\mathcal{M}(X)$ and $\tilde{f}^*\mathcal{M}(\tilde{X})$ are subfields of $\mathcal{M}(X_s)$ admitting the same degree of field extension. Moreover, there is only one subfield of index $m$. Hence, these two subfields coincide. As a consequence, $X$ and $\tilde{X}$ admit $\mathbb{C}$-algebra-isomorphic function fields (\cite{Giro,Ric}) and are isomorphic to one other.\\

The ramification points of $g$ along with their multiplicities are immediately known; hence, so too is the genus of $X$. Then, consider $tr$ many distinct points $$f(z_1),\dots ,f(z_{tr})\in X.$$ Now, $\text{mult}_\pi(z_i) = \text{mult}_f(z_i).\text{mult}_g(f(z_i)) \leq m.p = r$. Equality occurs if and only if $\text{mult}_f(z_i) = m$ and $\text{mult}_g(f(z_i)) = p$. Thus, each of these points is a ramification point under $g$. Moreover, these are the only ramification points of $g$ and they are mutually distinct. If $x_0\in X$ is a ramification point, choose $z$ in the (non-empty) fiber of $x_0$ under $f$. Then, $\text{mult}_\pi(z) = \text{mult}_f(z).\text{mult}_g(x_0) \geq 2$ i.e. $z$ is a ramification point of $\pi$. Thus $z = z_i$ for some $i$ and $x_0 = f(z_i)$. We therefore observe that $f(z_1),\dots , f(z_{tr})$ are the only ramification points of $g$, each possessing multiplicity $p$. Thus $b_g = tr(p - 1)$. From the Hurwitz formula, we obtain the genus of $X$.
\begin{example}
$g_X = \frac{tr(p - 1)}{2} + 1 - p$.
\end{example}

In the case $t = 2$ and $r = 4$, we obtain a hyper-elliptic curve $X$ of genus $3$. This is, once again, a unique curve in an appropriate sense.\\

Let us consider a $g^*(\mathcal{O}(t))$-twisted pair $(E',\phi': E'\to E'\otimes g^*\mathcal{O}(t))$ on $X$ such that $(g_*E', g_*\phi')$ is a stable Hitchin pair. Let suppose if possible that $(E',\phi')$ admits a nontrivial proper $\phi'$-invariant subbundle $F'$ such that $\mu_{F'} \geq \mu_{E'}$. Then $g_*F$ is a nontrivial proper $g_*\phi'$-invariant subbundle of $g_*E'$ such that $\mu_{g_*F'} \geq \mu_{g_*E'}$. This is a contradiction; hence, $(E',\phi')$ must be stable. The same thing holds with semistability. This is observed in the following diagram:
\begin{equation}\label{pushf}
\begin{tikzcd}
F' \arrow[r, "\phi'"] \arrow[d, "i"]
& F'\otimes g^*\mathcal{O}(t) \arrow[d, "i \otimes id"] \\
E' \arrow[r, "\phi' "]
& E'\otimes g^*\mathcal{O}(t)
\end{tikzcd}\xrightarrow{g_*} \begin{tikzcd}
g_*F' \arrow[r, "g_*\phi'"] \arrow[d, "i"]
& g_*F'\otimes \mathcal{O}(t) \arrow[d, "i \otimes id"] \\
g_*E' \arrow[r, "g_*\phi' "]
& g_*E'\otimes \mathcal{O}(t)
\end{tikzcd}
\end{equation}

Let $s\in H^0(\mathbb{P}^1,\mathcal{O}(tr))$ be a generic section and $(E, \phi)$ be a generic pair with characteristic polynomial $\lambda^r - s$ as in \ref{Spec}. From the factorization of $\pi = g\circ f$ there is a pair on $X$ namely $(f_*M, f_*\eta)$ and $(E, \phi) \cong g_*f_*(M, \eta) \cong g_*(f_*M, f_*\eta)$. Further we have that $[(M,\eta)]\mapsto [(f_*M, f_*\eta)]$ is an injective morphism into the space of isomorphism classes of $g^*\mathcal{O}(t)$-twisted pairs on $X$. Indeed $(f_*M, f_*\eta) \cong (f_*M', f_*\eta) \implies g_*(f_*M, f_*\eta) \cong g_*(f_*M', f_*\eta) \implies \pi_*(M,\eta) \cong \pi_*(M', \eta)$. From the correspondence in Remark \ref{spec1}, $M \cong M'$.\\ 

Let the set of the isomorphism classes of stable $g^*\mathcal{O}(t)$-twisted pairs of rank $m$ on $X$ be denoted by $\mathcal{N}$; similarly, denote by $\mathcal N'$ the collection of isomorphism classes of pairs inside $f_*\text{Pic} (X_s)$.  Continuing this way, we use $\mathcal N''$ to denote the set of isomorphism classes of $t$-twisted Hitchin pairs $(E,\phi)$ of rank $r$ on $\mathbb{P}^1$ with the characteristic equation $\lambda^r = s$. Note that the restricted pushforward morphism $g_*$ is immediately injective and surjective. The pair $(f_*M, f_*\eta)$ satisfies an equation $\lambda^r + g^*s_1\lambda^{r - 1} +\dots + g^*s_r = 0$.  In case of a cyclic cover, we have $s_1 = s_2 =\dots = s_{r - 1} = 0$ as components of a tuple in $\bigoplus_{i = 1}^rg^*(H^0(\mathbb{P}^1, \mathcal{O}(ti)))\subset\bigoplus_{i = 1}^rH^0(\mathbb{P}^1, g^*\mathcal{O}(ti))$. However, this is certainly not the characteristic polynomial of this pair on $X$; rather it is an annihilating polynomial. Informally, we name $\bigoplus_{i = 1}^rg^*(H^0(\mathbb{P}^1, \mathcal{O}(ti)))$ an \textit{iterated Hitchin base}.\\

 \begin{remark}
The degree of the shifted Jacobian on $X_s$, regarded as the degree of the line bundles it parametrizes, is$$d' = d + (g_{X_s} - 1) + r = d + (r - 1)(\frac{tr - 2}{2} + 1),$$where $d=\deg(E)$.  At the same time, we have $H^{-1}(s)\cong \text{Jac}^{d'}(X_s)$. From $\pi = g\circ f$ with $\deg(f) = m$ and $\deg(g) = p$, it follows that $\deg(f_*M) = d + \frac{mtr(p - 1)}{2}$.
\end{remark}

Although the existence of the iterated spectral covers is established, we lack any precise control over them. Let $\mathcal{J}$ denote the image $f_*(\mbox{Jac}^{d'}(X_s))$. We obtain a lower bound on Nitsure's dimension (cf. \cite{Nitin} Proposition 7.1) of the Zariski tangent space of stable $g^*\mathcal{O}(t)$ pairs on an intermediate (or iterated) spectral cover $X$, by a smooth embedding $f_*$ (if $f_*$ is a smooth embedding at all): it is at least the genus $g_{X_s}$ of the spectral curve $X_s$ due to containment of $\mathcal{J}$. As $\deg(g^*\mathcal{O}(t)) = tp$ we calculate $2(g_X - 1) - \deg(g^*\mathcal{O}(t)) = p(t(m(p - 1) - 1) - 2)$. This number is generally positive and $0$ at the base case $t = 2, r = 4$. So, there is lack of information to compute Nitsure's dimension in any of these cases $t\geq 2$ and $r\geq 4$. In this scenario, we are unable to comment if $\mathcal{J}$ is the full space $\mathcal{M}'_X(m, d'',g^*\mathcal{O}(t))$. We have no specific information about the compact Riemann surface $X$ apart from its abstract existence. So, it is also difficult to decide whether $X$ can be embedded within $\text{Tot}(\mathcal{O}(t))$.\\

The collection of the isomorphism classes of pairs $\mathcal{J}$ is identified with the elements of the Jacobian of $X_s$. We can refer to this object as an \textit{iterated Hitchin fiber}. The spectral correspondence of the line bundles and the pairs thus extends to a threefold correspondence. We organize the whole discussion in the form of the following theorem.

\begin{theorem}\label{finalth}
    Let $s\in H^0(\mathbb{P}^1, \mathcal{O}(tr))$ be a generic holomorphic section with $t\geq 2$ and $r$ be a composite number.\\

    \emph{\textbf{A:}} The isomorphism classes of $t$-twisted Hitchin pairs $(E,\phi)$ of rank $r$ on $\mathbb{P}^1$ satisfying the characteristic equation $\lambda^r = s$ (name this collection $\mathcal{N}''$) are in one-to-one correspondence with the isomorphism classes of line bundles $M$ on $X_s$. The correspondence is given with pushforward by the covering map $\pi$. In case we fix degree of $E$ to be $d\in\mathbb{Z}$, we see that $\pi_*: \emph{Jac}^{d'}(X_s) \to \mathcal{M}_{\mathbb{P}^1}(r, d, t)$ is a one-to-one correspondence, while $d' = d + (r - 1)(\frac{tr - 2}{2} + 1)$.\\
    
     \emph{\textbf{B:}} Given a factorization $r = mp$ with $p,m\geq 2$ there exists a compact Riemann surface $X$ and nonconstant holomorphic maps $f:X_s \to X$ of degree $m$ and $g: X\to \mathbb{P}^1$ of degree $p$ such that $\pi = g\circ f$. If there is another compact Riemann surface $\tilde{X}$ and nonconstant holomorphic maps $\tilde{f}:X_s \to \tilde{X}$ of degree $m$ and $\tilde{g}: \tilde{X}\to \mathbb{P}^1$ of degree $p$ such that $\pi = \tilde{g}\circ \tilde{f}$ then $X\cong \tilde{X}$.\\

     \emph{\textbf{C:}} Fix a chosen factorization of $r = mp$ and $\pi = g\circ f$. Let the space of isomorphism classes of stable $g^*\mathcal{O}(t)$-twisted pairs of rank $m$ on $X$, be $\mathcal{N}$. Then $f_*:\emph{Pic}(X_s)\to \mathcal{N}$ is a well-defined injective morphism with image $\mathcal{N}'$. There is a bijective correspondence $g_*:\mathcal{N}'\to \mathcal{N}''$. Given $\deg(E) = d$, the pushforward morphism given by $f_*: \emph{Jac}^{d'}(X_s) \to \mathcal{M}'_X(m, d'', g^*\mathcal{O}(t))$ is an injective morphism, wherein $\mathcal{M}'_X(m, d'', g^*\mathcal{O}(t))$ denotes the collection of the isomorphism classes of stable $g^*\mathcal{O}(t)$-twisted Hitchin pairs of rank $m$ and degree $d'' = d + \frac{mtr(p - 1)}{2}$ on $X$. Let $\mathcal{J}$ denote the image $f_*(\emph{Jac}^{d'}(X_s))$. Then $g_*:\mathcal{J}\to H^{-1}(s)$ is a bijective correspondence as $H$ is the Hitchin morphism on $\mathcal{M}_{\mathbb{P}^1}(r, d, \mathcal{O}(t))$.
    \end{theorem}
\begin{remark}
    Repeating our argument for a series of subgroups $0\subset H_1\subset H_2\subset\dots \subset H_k\subset G$, we decompose the covering map $\pi$ on ${X_s}$ in a polygonal series of iterated covers $X_1,.., X_k$, and we derive an intermediate series of twisted Hitchin pairs by iterating the composite projection formula. For example, on the $i$-th curve $X_i$ we may isolate a pair $(\beta(i)_*M, \alpha(i)^*L, \beta(i)_*\eta)$, expressed as per the convention of Equation \ref{mmoc} in which $\alpha(i) = f_{k+1}\circ\dots \circ f_{i+1}$ and $\beta(i) = f_i\circ..\circ f_1$. This discussion is captured in the following commutative diagram:
\end{remark}

\begin{equation}
    \begin{tikzcd}[column sep={1cm,between origins}, row sep={1.732050808cm,between origins}]
    & X_s \arrow[rr, "f_1"] \arrow[ld, "\pi"] \arrow[ddrr, phantom, "\circlearrowright"] && X_1 \arrow[rd, "f_2"] & \\
    \mathbb{P}^1 &  &&  & \vdots\arrow[ld, "f_{k-1}"'] \\
    & X_k \arrow[lu, "f_{k+1}"'] && X_{k-1} \arrow[ll, "f_k"'] & 
\end{tikzcd}
\end{equation}

\printbibliography

\end{document}